\documentclass[11pt,twoside,reqno]{amsart}

\usepackage{amsmath,amsfonts,amsthm,amssymb}
\usepackage{graphicx,psfrag,overpic}
\usepackage{subfigure}
\usepackage{caption2}
\usepackage{cite}
\usepackage{bm,enumerate}
\usepackage{cases}
\usepackage[all,cmtip,line]{xy}
\usepackage{array,CJK}
\usepackage{color}
\usepackage{booktabs}
\usepackage{underscore}
\usepackage{color}
\usepackage{appendix}
\usepackage{makecell}
\usepackage{pgfplots}
\usepackage{tikz}
\usepackage{pstricks,pst-node}
\usepackage{cleveref}

\numberwithin{equation}{section}
\numberwithin{figure}{section}
\numberwithin{table}{section}

\usepackage{graphics}
\usepackage{epsfig}
\newtheorem{claim}{\bf \t}[part]

\newtheorem{theorem}{Theorem}[section]

\newtheorem{lemma}{Lemma}[section]
\newtheorem{proposition}[theorem]{Proposition}
\newtheorem{remark}{Remark}[section]

\def\vep{\varepsilon}
\def\ep{\epsilon}
\def\pt{\partial}

\def\br{\bar{\rho}}
\def\bu{\bar{u}}

\def\div{\mathrm{div}}

\def\mI{\mathrm{I}}

\def\Du{\mathfrak{D}(u)}
\def\Dtu{\mathfrak{D}(\tu)}
\def\Dpsi{\mathfrak{D}(\psi)}

\def\tC{{\widetilde{C}}}

\def\ttau{\widetilde{\tau}}
\def\fT{\mathcal{T}}
\def\tfT{\widetilde{\fT}}
\def\Nhat{\widehat{N}}

\def\tW{\widetilde{W}}
\def\trho{{\widetilde{\rho}}}

\def\tu{\widetilde{u}}
\def\tthe{{\widetilde{\theta}}}

\def\mup{\mu^\prime}
\def\nup{\nu^\prime}
\def\Ma{\mbox{Ma}}
\def\Re{\mbox{Re}}
\def\Pr{\mbox{Pr}}

\def\bfT{\bar{\fT}}

\newcommand \R{\mathbb{R}}

\begin{document}

\title[]
{Global existence and low mach number limit of strong solutions to the full compressible Navier-Stokes equations around the plane Couette flow}
\author{Tuowei Chen}
\address{Tuowei Chen, Institute of Applied Physics and Computational Mathematics, 100088 Beijing, P. R. China;}
\email{tuowei\_chen@163.com}
\author{Qiangchang Ju}
\address{Qiangchang Ju, Institute of Applied Physics and Computational Mathematics, 100088 Beijing, P. R. China;}
\email{ju\_qiangchang@iapcm.ac.cn}

\date{\today}

\begin{abstract}
In this paper, we study the global existence and low Mach number limit of strong solutions to the 2-D full compressible Navier-Stokes equations around the plane Couette flow in a horizontally periodic layer with non-slip and isothermal boundary conditions. It is shown that the plane Couette flow is asymptotically stable for sufficiently small initial perturbations, provided that the Reynolds number, Mach number and temperature difference between the top and the lower walls are small. For the case that both the top and the lower walls maintain the same temperature, we further prove that such global strong solutions converge to a steady solution of the incompressible Navier-Stokes equations as the Mach number goes to zero. 
\end{abstract}

\maketitle

\section{Introduction}
This paper is concerned with the motion of viscous compressible gas flow between two parallel walls that are separated by a distance $h$, where the top wall moves with a constant speed $V_{1}$ and the lower wall is stationary, with the temperatures of the top and the lower walls being maintained at $\fT_1>0$ and $\fT_b>0$, respectively. The gas flow is governed by the full compressible Navier-Stokes equations
\begin{equation}
	\left.\begin{cases}
		\pt_t\rho+\div(\rho u)=0,\\
		\pt_t(\rho u)+ \div(\rho u\otimes u)+\nabla P= \div\mathbf{\tau},\\
		\pt_t(\rho E)+\div\big((\rho E+P) u\big)= \div q+\div(\mathbf{\tau} u)
	\end{cases}\label{2DCNS}
	\right.
\end{equation}
in a two-dimensional infinite layer $\Omega_{h}=\R\times(0,h)$:
\begin{align}
	\Omega_{h}=\{x=(x_1,x_2):x_1\in\R,\,0<x_2<h\}.
\end{align}
Here, $\rho(x,t)$, $u(x,t)=(u^1(x,t),u^2(x,t))^\top$ and $\fT(x,t)$ are unknowns and represent the gas density, velocity and absolute temperature, respectively, at time $t\geq0$, and position $x\in\Omega_h$; $P$ is the gas pressure and $E=e+\frac{1}{2}|u|^2$ is the specific total energy, where $e$ is the specific internal energy. $\mathbf{\tau}$ and $q$ are the viscous stress tensor and the heat flux vector, respectively, that are given by
\begin{align}
	\mathbf{\tau}=2\nu\Du+\nup\div u \mI,\,\,\,\,\,\, q=\lambda\nabla \fT,
\end{align} 
where $\Du:=\frac{1}{2}\big(\nabla u+ (\nabla u)^\top\big)$ is the deformation tensor, $\mI$ denotes the identity matrix, $\nu$ and $\nup$ are the viscosity coefficients that are assumed to be constants and satisfy
\begin{align}
\nu>0,\,\,\,\,\,\nu+\nup>0;
\end{align}
and $\lambda$ is the heat conductivity coefficient that is assumed to satisfy (cf. \cite{MDA2008,Xu2001})
\begin{align}
	\lambda =\frac{C_p\nu}{\Pr}.   \label{kappa}
\end{align}
In the above formula, $\Pr$ and $C_p$ are positive constants standing for the Prandtl number and the specific heat at constant pressure, respectively. We study the ideal polytropic gas so that $C_p$ satisfies
\begin{align}
	C_p=\frac{\gamma}{\gamma-1}R,
\end{align}
and $e$ and $P$ are given by the state equations
\begin{align}
	e=e(\fT)=\frac{R\fT}{\gamma-1}, \quad P=P(\rho,\fT)=R\rho\fT,
\end{align}
where constant $\gamma>1$ is the ratio of the specific heats and $R$ is the universal gas constant. Without loss of generality, in this paper we assume that $R=1$. 

The boundary conditions over the top and lower walls read 
\begin{equation}
	\left.\begin{cases}
	u^1(x_1,h,t)=V_{1}, \,\,\, &u^1(x_1,0,t)=0, \\
	u^2(x_1,h,t)=0, \,\,\,\,\,  &u^2(x_1,0,t)=0, \\
	\fT(x_1,h,t)=\fT_1, \,\,\, &\fT(x_1,0,t)=\fT_b.
	\end{cases}\label{boundarycon}
\right.
\end{equation}
We also require periodicity of $W:=(\rho, u,\fT)$ in $x_1$-direction:
\begin{align}
	W(x_1+2\pi h k,x_2,t)=	W(x_1,x_2,t),\,\,\,\,\,\forall k\in \mathbb{Z}. \label{boundarycon1}
\end{align}
Here, without loss of generality, the length of the basic periodic cell is set as $2\pi h$.

It is easily seen that the system \eqref{2DCNS}--\eqref{boundarycon1} has a steady solution $\overline{W}$ satisfying
\begin{equation}
\begin{aligned}
&\overline{W}=(\br,\bu,\bfT)^\top,\\
&\bfT=\fT_b+(\fT_1-\fT_b)\frac{x_2}{h}-\frac{\nu V_{1}^2}{2\lambda}\left[\left(\frac{x_2}{h}\right)^2-\frac{x_2}{h}\right],\\
&\br=p_{1}(\bfT)^{-1},\,\,\,\,\, \bu^1=V_{1}\frac{x_2}{h}, \,\,\,\,\, \bu^2=0,
\end{aligned}    \label{Couette}
\end{equation}
where $p_{1}$ is a positive constant standing for the steady pressure. $\overline{W}$ is driven by the speed and temperature differences between the top and the lower walls, and is known as the plane Couette flow; see \cite{MDA2008}. In this paper, we assume that $p_{1}=1$.

Based on the reference quantities from the steady solution $\overline{W}$ and the distance between the two walls, the Mach number $\Ma$ and Reynolds number $\Re$ are defined as
\begin{align}
	\Ma= \frac{V_{1}}{\sqrt{\gamma \fT_1}},\,\,\,\,\, \Re=\frac{\rho_{1}V_{1}h}{\mu},    \label{ReMa}
\end{align}
where
\begin{align}
	 \rho_{1}:=\br|_{x_2=h}=\frac{1}{\fT_1}.
\end{align}
Moreover, the temperature ratio between the lower and the top walls is defined as
\begin{align}
	\chi=\frac{\fT_b}{\fT_1}.    \label{Tdiffratio}
\end{align}
We point out that the plane Couette flow $\overline{W}$ has been widely used as a benchmark for low Mach number viscous and heat-conductive flows in the context of computational fluid dynamics; see \cite{Xu2001,KW2010,CYLLL2016}.

There have been many works on the plane Couette flow of incompressible fluids in the literature. It is well known that the plane Couette flow of the incompressible Navier–Stokes equation is in general stable under any initial perturbation in $L^2$ if the Reynolds number $Re$ is sufficiently small. At high Reynolds number regime, Romanov \cite{R1973} first proved that the plane Couette flow of the incompressible Navier–Stokes equations is stable for any Reynolds number $Re>0$ under sufficiently small disturbances. More recent progress on such incompressible plane Couette flows can be found in \cite{MZ2020,WZZ2018,WZZ2019,WZ2021} and the references therein.

Here we are mainly concerned with the plane Couette flow of the compressible Navier-Stokes equations. In the isentropic regime, Iooss-Padula \cite{IP1998} investigated the linear stability of stationary parallel flow in a cylindrical domain. When the nonlinear effect is concerned with, Kagei \cite{K2011} proved that the plane Couette flow in an infinite layer is asymptotically stable for sufficiently small initial disturbances if the Reynolds and Mach numbers are sufficiently small, and showed the asymptotic behavior of the disturbances as well. Later, Kagei \cite{K2011JDE,K2012} extended this result to general parallel flows. Li and Zhang \cite{LZ2017} studied the stability of plane Couette flow in an infinite layer with Navier-slip boundary condition at the bottom boundary. Readers are referred to \cite{KN2015,KN2019,KT2020} for further results on plane Poiseuille flows and Couette flows between two concentric cylinders. In the non-isentropic regime, Duck et al. \cite{DEH1994} investigated the linear stability of the plane Couette flow for the full compressible Navier-Stokes equations. Recently, Zhai \cite{Z2021} studied the linear stability of the Couette flow for the non-isentropic compressible Navier-Stokes equations with vanished shear viscosity. However, to the best of our knowledge, there is few result on the nonlinear stability of the plane Couette flow for the full compressible Navier-Stokes equations.

Physically, as the Mach number vanishes, the behaviors of compressible fluid flows would tend to the incompressible ones (cf. \cite{Lions1996}). Mathematically, this is a singular limit. The rigorous justification of this limit process has been studied extensively since the pioneer work by Klainerman and Majda \cite{KM1981,KM1982} for local strong solutions of compressible fluids (Naiver-Stokes or Euler). Here we focus more specifically on the study of the low Mach number limit of the compressible Navier-Stokes equations. Alazard \cite{A2006} studied the low Mach number limit of the full Navier-Stokes equations in the whole space. For boundary-value problems, the researches of low Mach number limit problem in a bounded domain with the slip type boundary conditions are quite fruitful. For the case with vorticity-slip boundary conditions, see the studies by Ou \cite{O2009}, by Ou and Yang \cite{OY2022}, and by Ju and Ou \cite{JO2022}. For the case with Navier-slip boundary conditions, see the works by Ren and Ou \cite{RO2016}, by Masmoudi et al. \cite{MRS2022}, and by Sun \cite{S2022}. 

It is particularly interesting and more difficult to study the low Mach number problem in a bounded domain with non-slip boundary conditions, where the terms containing normal derivatives of the velocity need to be treated carefully due to boundary effects. Bessaih \cite{B1995} studied the low Mach number behavior of local strong solutions to the compressible Navier-Stokes equations in a bounded domain with non-slip boundary conditions. Later, Jiang and Ou \cite{JO2011} extended this result to the case for the non-isentropic Navier-Stokes equations with zero thermal conductivity coefficient. Related studies in the weak solution framework can be found in \cite{LM1998,DGLM1999}. Very recently, Fan et al. \cite{FJX2023} proved the long time existence of the slightly compressible isentropic Navier-Stokes equations in bounded domains with non-slip boundary conditions. We comment that the above-mentioned studies of strong solutions are based on an asymptotic expansion for the solution around a motionless constant state, i.e., $(\rho,u,T)=(1,0,1)$. It is more attractive to study the low Mach limit of solutions to the compressible Navier-Stokes equations around a steady flow with nonzero velocity, e.g., the Plane Couette flow \eqref{Couette}. We also note that the Plane Couette flow \eqref{Couette} is an important exact steady solution to the full compressible Navier-Stokes equations, where both velocity and temperature enjoy inhomogeneous Dirichlet type boundary conditions. To the best of our knowledge, there is no result on the low Mach limit of solutions to the full compressible Navier-Stokes equations around the Plane Couette flow. In addition, we mention an interesting work by Huang et al. \cite{HWW2017} that the solutions of 1-D full compressible Navier–Stokes equations with different end states converge to a nonlinear diffusion wave solution globally in time as the Mach number goes to zero.

The purpose of this paper is twofold: (i) to study the global in time stability of the Couette flow $\overline{W}$ with small initial perturbations and the small Mach number; (ii) to study the low Mach number limit of the global strong solutions around the Couette flow $\overline{W}$.

To this end, we derive a non-dimensional form of system \eqref{2DCNS}--\eqref{boundarycon1} as follows. We introduce the non-dimensional variables:
\begin{equation}
	x=h\hat{x}, \,\, t=\frac{h}{V_{1}}\hat{t}, \,\, 
	u=V_{1}\hat{u},\,\,  \rho=\rho_{1}\hat{\rho},\,\, 
	\fT= T_{1}\hat{\fT},   \label{transform}
\end{equation}
Under the transformation \eqref{transform}, system \eqref{2DCNS} on $\Omega_{h}$ with boundary conditions \eqref{boundarycon}--\eqref{boundarycon1} are written, by omitting hats, as
\begin{equation}
	\left.\begin{cases}
		\pt_t\rho+\div(\rho u)=0,\\
		\rho(\pt_t u+ u\cdot\nabla u)+\ep^{-2}\nabla P(\rho,\fT)= \mu\Delta u+(\mu+\mup)\nabla\div u,\\
		\frac{\rho}{\gamma-1}(\pt_t \fT+ u\cdot\nabla \fT)+P(\rho,\fT)\div u
		=\kappa\Delta\fT+\ep^2\big(2\mu|\Du|^2+\mup(\div u)^2\big),
	\end{cases}\label{2DCNS0}
	\right.
\end{equation}
on a horizontally periodic domain $\Omega=\R/(2\pi\mathbb{Z})\times (0,1)$:
\begin{align}
	\Omega=\{x=(x_1,x_2):x_1\in\mathcal{S}^1,\,0<x_2<1\},   \label{OmegaP}
\end{align}
subjected to the boundary condition
\begin{equation}
	\left.\begin{cases}
		u^1(x_1,1,t)=1, \,\,\, &u^1(x_1,0,t)=0, \\
		u^2(x_1,1,t)=0, \,\,\,\,\,  &u^2(x_1,0,t)=0,\\
		\fT(x_1,1,t)=1, \,\,\, &\fT(x_1,0,t)=\chi, \quad\quad\quad  \forall x_1\in \mathcal{S}^1,\, t>0,
	\end{cases}\label{boundarycon2}
	\right.	
\end{equation}
and the initial condition
\begin{align}
	W(x,0)=W_0(x)=(\rho_0(x),u_0(x),\fT_0(x))^\top, \quad\quad \forall x\in \Omega. \label{initialcon1}
\end{align}
Here, and in the sequel, $\mathcal{S}^1$ denotes the unit circle, and $\ep$, $\mu$, $\mup$ and $\kappa$ are non-dimensional parameters given by  
\begin{equation}
	\ep=\sqrt{\gamma}\Ma, \,\, \mu=\frac{1}{\Re}, \,\, \mup=\frac{\nup}{\Re\nu},\,\, \kappa=\frac{C_p}{\Re \Pr}.  \label{ReMAnew}
\end{equation}
To derive \eqref{2DCNS0}, we have used the relation \eqref{kappa} and \eqref{ReMa}. Accordingly, the Couette flow \eqref{Couette} is transformed to 
\begin{align}
	\tW=(\trho,\tu,\tfT)^\top,
\end{align}
where
\begin{equation}
	\tfT=\chi +(1-\chi)x_2-\frac{\ep^2\Pr }{2C_p}(x^2_2-x_2),\,\,\,\,\, \tu^1=x_2, \,\,\,\,\, \tu^2=0,\,\,\,\,\,\trho=(\tfT)^{-1}.\label{Couette1}
\end{equation}


Before stating the main theorems, we introduce the notations used in this paper. \\
$\bullet\,\,\,L^q(\Omega)$: The standard Lebesgue space over $\Omega$ with the norm $\| \cdot \|_{L^{q}}\,\,(1\leq q\leq \infty)$.\\
$\bullet\,\,\,H^l(\Omega)$: The usual $L^2$-Sobolev space over $\Omega$ of integer order $l$ with the norm $\| \cdot \|_{H^{l}}\,\,(l\geq 0)$.\\
$\bullet\,\,\,C([0,T];H^l(\Omega))$: The space of continuous functions on an interval $[0,T]$ with values in $H^l(\Omega)$.\\
The function spaces $L^2([0,T];H^1(\Omega))$ and $L^\infty([0,T];H^2(\Omega))$ can be defined similarly.

For a function $f$, we use the simplified notation
\begin{align}
	\int f dx:=\int_{\Omega} f dx.
\end{align}
We also denote by $\pt_{1}$ and $\pt_{2}$ the operators $\pt_{x_1}$ and $\pt_{x_2}$, respectively. Moreover, the  perturbation functions are defined by
\begin{align}
     (\varphi(t),\psi(t),\theta(t))^\top:=(\rho(t)-\trho,u(t)-\tu,\fT(t)-\tfT)^\top.  \label{def}
\end{align}

The main theorem of this paper can be stated as follows.
\begin{theorem} \label{main1}	
		Suppose that $\chi$, i.e, the temperature ratio between the lower and the top walls, satisfies
	\begin{align}
		|1-\chi|=O(\ep).  \label{chi}
	\end{align}
	Suppose that the initial perturbation $(\varphi_0,\psi_0,\theta_0):=(\varphi(0),\psi(0),\theta(0))$ satisfies
	\begin{equation}
	\begin{aligned}
		&\varphi_0\in H^3(\Omega),\quad \int_{\Omega}\varphi_0(x) dx=0,\\
		&\psi_0\in  H^1_0(\Omega)\cap H^3(\Omega),
		\quad \theta_0\in  H^1_0(\Omega)\cap H^3(\Omega).
	\end{aligned}   \label{compatibilitycon}
\end{equation}
Assume further that the compatibility conditions are satisfied:
\begin{equation}
	\left.\begin{cases}
	\pt_{t}\psi(0)=0,\quad\quad\text{on}\quad\pt\Omega, \\
	\pt_{t}\theta(0)=0,\quad\quad\text{on}\quad\pt\Omega.
	\end{cases}
\right.	 \label{compatibilitycon1}
\end{equation}
Then, there exist positive constants $\Re^\prime$, $\vep^\prime$ and $C_0$ such that if
\begin{align}
	\Re<\Re^\prime, \quad \ep<\vep^\prime,  \label{con2}
\end{align}
and if
\begin{equation}
	\begin{aligned}
		&\|\psi(0)\|^2_{L^2}+\ep^{2}\|\nabla\psi(0)\|^2_{L^2}
		+\ep^{4}\|(\nabla^2\psi,\pt_{t}\psi)(0)\|^2_{L^2}
		+\ep^{6}\|(\nabla^3\psi,\nabla\pt_{t}\psi)(0)\|^2_{L^2}\leq C_0\ep^2,\\
			&\|\varphi(0)\|^2_{L^2}+\ep^{2}\|\nabla\varphi(0)\|^2_{L^2}
		+\ep^{4}\|(\nabla^2\varphi,\pt_{t}\varphi)(0)\|^2_{L^2}
		+\ep^{6}\|\nabla^3\varphi(0)\|^2_{L^2}\leq C_0\ep^4,\\
			&\|\theta(0)\|^2_{L^2}+\ep^{2}\|\nabla\theta(0)\|^2_{L^2}
		+\ep^{4}\|(\nabla^2\theta,\pt_{t}\theta)(0)\|^2_{L^2}
		+\ep^{6}\|(\nabla^3\theta,\nabla\pt_{t}\theta)(0)\|^2_{L^2}\leq C_0\ep^4,
	\end{aligned}  \label{con3}
\end{equation}
then there exists a unique global strong solution $(\rho,u,\fT)$ to \eqref{2DCNS0}--\eqref{initialcon1} satisfying 
\begin{equation}
	\begin{aligned}
		&\rho\in C\big([0,\infty);H^3(\Omega)\big),
		\quad u\in C\big([0,\infty);H^3(\Omega)\big)\cap L^2\big([0,\infty);H^4(\Omega)\big),\\
		&\fT\in C\big([0,\infty);H^3(\Omega)\big)\cap L^2\big([0,\infty);H^4(\Omega)\big),
	\end{aligned}   \label{regularity}
\end{equation}
and
\begin{equation}
	\begin{aligned}
		&\sup_{t\in [0,\infty)}\big(\ep^{-1}\|\varphi(t)\|_{L^2}+\|\psi(t)\|_{L^2}
		+\ep^{-1}\|\theta(t)\|_{L^2}\big)\leq \widehat{C}\ep,\\
		&\sup_{t\in [0,\infty)}\big(\ep^{-1}\|\nabla\varphi(t)\|_{L^2}
		+\|\nabla\psi(t)\|_{L^2}+\ep^{-1}\|\nabla\theta(t)\|_{L^2}\big)	\leq \widehat{C}.
	\end{aligned}  \label{uniform}
\end{equation}
Here, $\widehat{C}$ is a positive constant independent of $\ep$. Moreover, we have
\begin{align}
	\|\big(\rho(t)-\trho,u(t)-\tu,\fT(t)-\tfT\big)\|_{L^\infty}\rightarrow 0 
	\quad\quad \text{as} \quad t\rightarrow \infty. \label{largetime}
\end{align}
Therefore, the plane Couette flow \eqref{Couette1} with the initial perturbation $(\varphi_0,\psi_0,\theta_0)$ is asymptotically stable. 
\end{theorem}

\begin{remark}
	It follows from \eqref{Couette1} and \eqref{def} that $\pt_{t}\psi(0)=\pt_{t}u(0)$. The notation $\pt_{t}u(0)$ is defined by taking $t = 0$ in \eqref{eq2}, that is,
	\begin{align}
		\rho_0(\pt_t u(0)+ u_0\cdot\nabla u_0)+\ep^{-2}\nabla P(\rho_0,\fT_0)= \mu\Delta u_0+(\mu+\mup)\nabla\div u_0.
	\end{align}
It is indeed a compatibility condition. The notations $\pt_{t}\varphi(0)$ and $\pt_{t}\theta(0)$ are defined in a similar way. 
\end{remark}

\begin{remark}
	The condition $\eqref{compatibilitycon}_1$ naturally follows from the conservation of mass.
\end{remark} 

\begin{remark}
From proof below, the constant $\ep^\prime$ indeed depends on the Reynolds number $\Re$ in \Cref{main1}. Moreover, the constant $C_0$ depends on $\Re$, but is independent of $\ep$.  
\end{remark} 

\begin{remark}
	Note that for the case that $\chi=1+\frac{\ep^2\Pr }{2C_p}$, the temperature of the Couette flow $\tW$ reads 
	\begin{equation}
		\tfT=\chi-\frac{\ep^2\Pr }{2C_p}x^2_2.
	\end{equation}
	In this case, the isothermal lower wall corresponds to an adiabatic lower wall, i.e., 
	\begin{align}
		\pt_{x_2}\fT|_{x_2=0}=0.
	\end{align}
	It can be proved in a similar manner that with the above Neumann boundary condition for the temperature at the lower wall,  Theorem \ref{main1} also holds. The details are omitted here. 
\end{remark}

For the case that there is no temperature difference between the top and the lower walls, i.e.
\begin{align}
	\chi=1,  \label{same}
\end{align}   
which can be covered by the condition \eqref{chi}, we state the result on low Mach number limit of the global strong solutions obtained in \Cref{main1} as follows. 
\begin{theorem} \label{main2}
Suppose that \eqref{same} holds, and suppose that the condition \eqref{compatibilitycon}--\eqref{con3} hold as in \Cref{main1}. Denote by $(\rho^\ep,u^\ep,\fT^\ep)$ the unique global strong solution obtained from \Cref{main1}. Then, it holds that
\begin{align}
	\|\rho^\ep-1\|_{L^2}=O(\ep^2),\quad	\|u^\ep-\tu\|_{L^2}=O(\ep),\quad	\|\fT^\ep-1\|_{L^2}=O(\ep^2). \label{converge0}
\end{align} 
Note that $(\tu,p_*)$ with $p_*:= P(1,1)$ is indeed the incompressible plane Couette flow being a steady solution to the following initial-boundary value problem 
 \begin{equation}
	\left.\begin{cases}
		\div v=0,    \quad\quad\text{in}\quad \Omega\times[0,\infty),  \\
		\pt_t v+ v\cdot\nabla v+\nabla p = \mu\Delta v,\quad\quad\text{in}\quad \Omega\times[0,\infty),\\
         v|_{x_2=1}=1,\quad v|_{x_2=0}=0,\\
         v|_{t=0}=\tu.
	\end{cases}\label{2DINS}
	\right.
\end{equation} 
Thus, the global strong solution $(\rho^\ep,u^\ep,\fT^\ep)$ converges to $(1,\tu,1)$, which gives the incompressible plane Couette flow, as $\ep\rightarrow0$.
\end{theorem}

\begin{remark}
	We mention that \Cref{main1,main2} can be extended to the case of 3-D full compressible Navier-Stokes equations around the plane Couette flow by slight modifications.  
\end{remark}

Now we make some comments on the analysis of this paper. Compared to Matsumura-Nishida's results \cite{MN1980,MN1983} on global existence of strong solutions to the full compressible Navier-Stokes equations with small initial disturbance around a motionless constant state, and compared to the results (e.g., \cite{JO2011,JO2022}) on low Mach number limit of compressible Navier-Stokes equations around a motionless constant state, the Couette flow \eqref{Couette1} with non-zero velocity brings some essential difficulties. More specifically, we need control the convective effects brought by the non-zero velocity $\tu$. Moreover, to study the low Mach number behavior of the solutions, we need to derive the uniform estimates independent of $\ep$ and the time. However, due to the non-slip boundary conditions, it is difficult to obtain the uniform estimates on normal derivatives of the velocity by using the method of integration by parts as in the case with Navier-slip boundary conditions \cite{O2009,RO2016,JO2022,OY2022,MRS2022,S2022}. Similarly, the isothermal boundary condition also cause troubles in applying the method of integration by parts as in the case with zero thermal conductivity coefficient \cite{JO2011}, the case with adiabatic boundary condition \cite{JO2022}, and the case with convective boundary condition \cite{RO2016}. In addition, quite different from the plane Couette flow of the isentropic compressible Navier-Stokes equations studied by Kagei \cite{K2011,K2012}, the plane Couette flow of the full compressible Navier-Stokes equations is more complicated, since both its density and temperature are no more constant when there is a temperature difference between the top and the lower walls. To overcome these difficulties, we consider the equations of perturbation and establish the uniform estimates on an energy functional which includes the $\ep$-weighted mixed time and spatial derivatives of the solutions. As one of the key ingredients in our proofs, based the observation that $\div\tu=0$, we employ the relative entropy method to derive the $\ep$-weighted basic energy estimate without involving high-order derivatives terms on the right hand side of the inequality, provided that some smallness assumptions hold. We also remark that to control the convective effect terms brought by $\tu$, we need to use the structure that $\div\tu=0$ and employ the $\ep$-weighted $L^2_TH^1$-type estimates on the perturbations. Based on the Poincar\'{e} inequality, the $L^2_TH^1$-type estimates of $\psi$ and $\theta$ naturally follow from the viscosity and heat conductivity effects, respectively. With help of the zero average assumption on the initial density perturbation $\eqref{compatibilitycon}_1$, we use the elliptic estimates on stationary Stokes equations in the spirits of \cite{MN1983,K2011} to obtain the $\ep$-weighted $L^2_TH^1$-type estimate on $\varphi$.  

The rest of this paper is organized as follows. In Section 2, we collect some basic facts and elementary inequalities. In Section 3, we show the uniform \textit{a priori} estimates independent of the Mach number by using a weighted energy method. Finally, the main theorems are proved in Section 4.

\section{Preliminaries}
In this section we first recall some known facts and elementary inequalities that will be used later. Later in the section, we give the system of equations for the perturbation and show the local existence and uniqueness of strong solutions to the initial-boundary value problem of the resulting system.

\subsection{Some elementary inequalities}
In this subsection we recall some known facts and elementary inequalities that will be used later.

We first recall the following Poincar\'{e} type inequality.
\begin{lemma} \label{Poincareineq2}
	Let $\Omega^b\subset\R^n$ $(n>1)$ be bounded and locally Lipschitz and let $\Sigma$ be an arbitrary portion of $\pt\Omega^b$ of positive (($n-1$)-dimensional) measure. Then, the following statements hold.
	
	\noindent(1) For $f\in W^{1,q}$ with $\int_{\Omega^{b}}fdx=0$, $1\leq q<\infty$, it holds that
	\begin{align}
		\|f\|_{L^q(\Omega^b)}
		\leq c_1\|\nabla f\|_{L^q(\Omega^b)},
	\end{align}
	where $c_1$ is a positive constant depending only on $n$, $q$ and $\Omega^b$;
	
	\noindent(2) For $f\in W^{1,q}$, $1\leq q<\infty$, the following inequality holds:
	\begin{align}
			\|f\|_{L^q(\Omega^b)}
			\leq c_2\left(\|\nabla f\|_{L^q(\Omega^b)}+\big|\int_{\Sigma}f\big|\right),
	\end{align}
	where $c_2$ is a positive constant depending only on $n$, $q$, $\Omega^b$ and $\Sigma$.
\end{lemma}

\begin{proof}
	See the Theorem II.5.4 and Exercise II.5.13 in \cite{Galdi2011}.
\end{proof}
The following is the Gagliardo-Nirenberg inequality which will be used frequently.
\begin{lemma}\label{GN lemma}
	For $p\in[2,\infty)$, there exists a generic positive constant $C$ which may depend on $p$ and $\Omega$, such that for any $f\in H^2(\Omega)$, we have
	\begin{align}
		&\|f\|_{L^p}    
		\leq C\|f\|^{\frac{2}{p}}_{L^2}\|\nabla f\|^{1-\frac{2}{p}}_{L^2}+C\|f\|_{L^2},\\
		&\|f\|_{L^\infty}\leq C\|f\|^{\frac{1}{2}}_{L^2}
		\|\nabla^2 f\|^{\frac{1}{2}}_{L^{2}}+C\|f\|_{L^2}.
	\end{align}
\end{lemma}
\begin{proof}
   See \cite{N1959}.
\end{proof}

Next, we recall the classical elliptic theorey for the Lam\'{e} system.
\begin{lemma}\label{Lame lemma}
	Let $u$ be a smooth solution solving the problem
	\begin{equation}
		\left.
		\begin{cases}
			\mu\Delta u+(\mu+\mup)\nabla\div u=F\quad\quad\text{in}\,\,\Omega,\\
			u=0\quad\quad\text{on}\,\,\pt\Omega.
		\end{cases} \label{lamesys}
		\right.
	\end{equation} 
    Then, for $p\in[2,\infty)$ and $k\in \mathcal{N}^{+}$, there exists a positive constant $C$ depending only on $\mu$, $\mup$, $p$, $k$ and $\Omega$ such that the following estimates hold:
    
	\noindent(1) if $F\in L^p(\Omega)$, then
	\begin{align}
		\|u\|_{W^{k+2,p}}\leq C\|F\|_{W^{k,p}};
	\end{align}
	\noindent(2) if $F=\div f$ with $f=(f^{ij})_{3\times 3}$, $f^{ij}\in W^{k,p}(\Omega)$, then
	\begin{align}
		\|u\|_{W^{k+1,p}}\leq C\|F\|_{W^{k,p}}.
	\end{align}
\end{lemma}
\begin{proof}
	See the standard $L^p$-estimates for the Agmon-Douglis-Nirenberg systems \cite{ADN1959,ADN1964}.
\end{proof}

Furthermore, we consider the following stationary nonhomogeneous Stokes equations,
\begin{equation}
	\left.
	\begin{cases}
	\div u=g\quad\quad\text{in} \,\,\Omega,\\
	-\mu\Delta u+\nabla p=F \quad\quad\text{in} \,\,\Omega,\\
    u=0 \quad\quad\text{on}\,\,\pt\Omega,
	\end{cases}\label{Stokes}
	\right.
\end{equation}
where $g$ and $F$ are known functions. We recall some known results for the Stokes system \eqref{Stokes}.
\begin{lemma}\label{Stokes lemma}
	If $F\in H^{k-1}(\Omega)$, $g\in H^k(\Omega)$ for some nonnegative integer $k$, and if the compatibility condition $\int_\Omega gdx=0$ holds, then there exists a unique solution $(u,p)$ in the space $H^{k-1}(\Omega)\times H^{k}_\#(\Omega)$ to the problem \eqref{Stokes}, where $H^{k}_\#(\Omega)\triangleq\big\{f\in H^{k}(\Omega):\int_\Omega fdx=0\big\}$ and $H^{-k}(\Omega)$ is the dual of $H^k_0(\Omega)$, the closure of $C^\infty_0(\Omega)$ in $H^{-k}(\Omega)$. Moreover,
	\begin{align}
		\mu\| u\|_{H^{k+1}}+\| p\|_{H^k}\leq C(\Omega,k)\big(\| F\|_{H^{k-1}}+\mu\| g\|_{H^k}\big), \label{Stokes1}
	\end{align}
	where $C(\Omega,k)$ is a positive constant depending at most on $\Omega$ and $k$.
\end{lemma}
\begin{proof}
	See \cite{Veigo1997} and the Chapter four in \cite{Galdi2011}.
\end{proof}

\subsection{System of equations for the perturbation}
In this subsection we derive the system of equations for the perturbation, and then show the local existence and uniqueness of strong solutions to the initial-boundary value problem of the resulting system. 

Recalling the notation \eqref{def}, we rewrite the system \eqref{2DCNS} as
 \begin{align}
	&\pt_t\varphi+u\cdot\nabla\varphi+\rho\div\psi+\psi\cdot\nabla\trho=0,   \label{eq1}\\
	&\rho\big(\pt_t\psi+u\cdot\nabla \psi+\psi\cdot\nabla\tu\big)
	+\ep^{-2}\nabla P(\rho,\fT)-\mu\Delta\psi-(\mu+\mup)\nabla\div\psi=0,\label{eq2}\\
	&\frac{1}{\gamma-1}\rho\big(\pt_t\theta+ u\cdot\nabla\theta+\psi\cdot\nabla\tfT\big)
	+ P(\rho,\fT)\div\psi-\kappa\Delta\theta\notag\\
	=& \ep^2\big(2\mu\Dpsi:\Dpsi+4\mu\Dtu:\Dpsi+\mup(\div\psi)^2\big).\label{eq3}
\end{align}

We consider \eqref{eq1}--\eqref{eq2} under boundary conditions
\begin{align}
	\psi|_{x_2=0,1}=0,\,\,\,\theta|_{x_2=0,1}=0, \label{eq4}
\end{align}
and the initial condition
\begin{align}
	(\varphi,\psi,\theta)^\top(0)=(\varphi_0,\psi_0,\theta_0)^\top.  \label{eq5}
\end{align}

We state the local existence result as follows:
\begin{proposition}[Local existence and uniqueness]\label{local}
	Suppose that \eqref{compatibilitycon}--\eqref{compatibilitycon1} hold. Then, there exists a small time $T$ such that there exists a local unique strong solution $(\varphi,\psi,\theta)$ to \eqref{eq1}--\eqref{eq5} on $\Omega\times[0,T]$ satisfying
	\begin{align}
		\varphi\in C([0,T];H^3(\Omega)),\,\,\,\,\,
		\psi,\theta\in C([0,T];H^3(\Omega))\cap L^2(0,T;H^4(\Omega)).\label{local0}
	\end{align}
\end{proposition}

\begin{proof}
The proof of this proposition can be done by using the linearization technique, classical energy method and Banach fixed point argument as in \cite{MN1980,MN1983}. For simplicity, the details are omitted here.  
\end{proof}

\section{Uniform estimates}
In this section, we derive the uniform estimates (in time) for smooth solutions to the initial-boundary value problem \eqref{eq1}--\eqref{eq5}.

In the following of this section, we fix a smooth solution $(\varphi,\psi,\theta)$ to \eqref{eq1}--\eqref{eq5} on $\Omega\times[0,T]$ for a time $T>0$, and assume that the conditions \eqref{chi}--\eqref{compatibilitycon1} hold. For $0\leq t \leq T$, we define 
\begin{equation}
\begin{aligned}
A_0(t):=&\sup_{s\in [0,t]}\big(\|\psi(s)\|^2_{L^2}+\ep^{-2}\|\varphi(s)\|^2_{L^2}+\ep^{-2}\|\theta(s)\|^2_{L^2}\big)\\
       &+\int_{0}^{t}\big(\|\psi(s)\|^2_{H^1}+\ep^{-2}\|\theta(s)\|^2_{H^1}\big)ds,\\
A_1(t):=&\sup_{s\in [0,t]}\big(\|\nabla\psi(s)\|^2_{L^2}+\ep^{-2}\|\nabla\theta(s)\|^2_{L^2}  \big)
+\int_{0}^{t}\big(\|\pt_t\psi\|^2_{L^2}+\ep^{-2}\|\pt_t\theta\|^2_{L^2}\big)ds,\\
A_2(t):=&\ep^{-2}\sup_{s\in [0,t]}\|\nabla\varphi(s)\|^2_{L^2}
+\int_{0}^{t}\big(\ep^{-4}\|\varphi\|^2_{H^1}
+\ep^{-2}\|\pt_{t}\varphi\|^2_{L^2}\big)ds\\
&+\int_{0}^{t}\big(\|\psi(s)\|^2_{H^2}+\ep^{-2}\|\theta(s)\|^2_{H^2}\big)ds,\\
A_3(t):=&\sup_{s\in[0,t]}\big(\ep^{-2}\|\pt_{t}\varphi(s)\|^2_{L^2}
+\|\pt_{t}\psi(s)\|^2_{L^2}+\ep^{-2}\|\pt_{t}\theta(s)\|^2_{L^2}  \big)\\
&+\int_{0}^{t}\big(\|\nabla\pt_{t}\psi(s)\|^2_{L^2}
+\ep^{-2}\|\nabla\pt_{t}\theta(s)\|^2_{L^2}\big)ds,\label{A3}\\
A_4(t):=&\sup_{s\in[0,t]}\big(\ep^{-2}\|\nabla^2\varphi(s)\|^2_{L^2}
+\|\nabla^2\psi(s)\|^2_{L^2}+\ep^{-2}\|\nabla^2\theta(s)\|^2_{L^2}  \big)\\
&+\int_{0}^t\big(\ep^{-4}\|\varphi(s)\|^2_{H^2}+\|\psi(s)\|^2_{H^3}
+\ep^{-2}\|\theta(s)\|^2_{H^3}\big)ds,\\
A_5(t):=&\sup_{s\in[0,t]}\big(\ep^{-2}\|\nabla^3\varphi(s)\|^2_{L^2}
+\|(\nabla^3\psi,\nabla\pt_{t}\psi)(s)\|^2_{L^2}+\ep^{-2}\|(\nabla^3\theta,\nabla\pt_{t}\theta)(s)\|^2_{L^2} \big)\\
&+\int_{0}^{t}\big( \|\pt^2_{t}\psi(s)\|^2_{L^2}+\ep^{-2}\|\pt^2_{t}\theta(s)\|^2_{L^2}\big)ds,
\end{aligned}
\end{equation}
and
\begin{align}
N(t):=&\ep^{-2}A_0(t)+A_1(t)+A_2(t)+\ep^{2}A_3(t)+\ep^{2}A_4(t)+\ep^{4}A_5(t).\label{N}
\end{align}
For simplicity, we use the notation that
\begin{align}
	\Nhat:=N(T),   \notag
\end{align}
and denote by $C$ a generic positive constant depending only on $\Re$, $\Pr$, $\frac{\mup}{\mu}$ and $\gamma$, but not on $\ep$ and $T$. In addition, we denote by $\tC$ a generic positive constant depending only on $\Pr$, $\frac{\mup}{\mu}$ and $\gamma$, but not on $\Re$, $\ep$ and $T$.

Moreover, recalling \eqref{chi}, since we focus on the case with low Mach number and small perturbations, we always assume that
\begin{align}
	\ep\leq1,\quad |1-\chi|\leq \frac{1}{2},  \quad \Nhat\leq 1.  \label{focus}
\end{align}	 

The main result of this section can be concluded as follows.
\begin{proposition} \label{priori}
	Suppose that \eqref{chi}--\eqref{compatibilitycon1} holds. Then, there exists a positive constant $\Re^\prime$ depending only on $\Pr$, $\frac{\mup}{\mu}$ and $\gamma$, and positive constants $\vep^\prime$ and $N^\prime$ depending only on $\Re$, $\Pr$, $\frac{\mup}{\mu}$ and $\gamma$, such that if 
 \begin{align}
	\Re<\Re^\prime, \quad \ep<\vep^\prime, \quad \Nhat <N^\prime,
\end{align}
    then it holds that
    	\begin{align}
    	\Nhat\leq \hat{C}N(0).  
    \end{align}   
   Here, $\hat{C}$ is a positive constant depending only on $\Re$, $\Pr$, $\frac{\mup}{\mu}$ and $\gamma$.
\end{proposition}
\begin{proof}
	 \Cref{priori} comes from \Cref{basic}, \Cref{lemmaH1}, \Cref{lemmaH2,lemmaH3} below.                           
\end{proof}

\subsection{Basic energy estimate}
This subsection is devoted to establishing the basic energy estimate for $A_0(t)$. 

We start with the following Poincar\'{e} type inequality. 
\begin{lemma}  \label{Poin}
	Suppose that the compatibility conditions \eqref{compatibilitycon1}--\eqref{compatibilitycon1} holds. Then, the following Poincar\'{e} type inequalities hold:  
	\begin{align}
		\|\psi(t)\|_{L^2}\leq \tC \| \nabla\psi(t)\|_{L^2},\,\,\,\,\, \|\theta(t)\|_{L^2}\leq \tC \| \nabla\theta(t)\|_{L^2},
		\,\,\,\,\,\,\, \forall \, t \in [0, T], \label{Poincare}
	\end{align}	
and 
 	\begin{align}
 		\|\varphi(t)\|_{L^2}
 	\leq \tC \| \nabla\varphi(t)\|_{L^2},
 	\,\,\,\,\,\,\, \forall \, t \in [0, T]. \label{Poincarerho}
 \end{align}	
 
\end{lemma}
\begin{proof}
	First, \eqref{Poincare} follows directly from the boundary condition \eqref{eq4} and \Cref{Poincareineq2}. 
	
	Next, it follows from \eqref{eq1} and \eqref{compatibilitycon} that the conservation of mass reads
	\begin{align}
		\int\rho(t)dx=	\int\rho_0 dx=\int\trho dx.   \notag
	\end{align}
	This yields that
	\begin{align}
		\int\varphi(t)dx=0.   \notag
	\end{align}
	Consequently, we obtain \eqref{Poincarerho} from \Cref{Poincareineq2}.	
	
	The proof is completed.
\end{proof}

Next, we show some elementary observations which will be used frequently.
\begin{lemma}\label{elementary}
	Suppose that \eqref{chi} holds. Then, there exists a positive constant $\vep_0$ depending only on $\Pr$, $\gamma$ and $\frac{\mup}{\mu}$, such that if $\ep<\vep_0$, then	
		\begin{align}
			\inf_{x\in\Omega}\trho(x)>\frac{3}{4},	\quad\inf_{x\in\Omega}\tfT(x)>\frac{3}{4},
			\quad \|\trho-1\|_{L^\infty\cap H^4}\leq \tC\ep,
			\quad \|\tfT-1\|_{L^\infty \cap H^4}\leq \tC\ep,      \label{ob0}
		\end{align}
	and
	\begin{equation}
	\begin{aligned}
		&\inf_{(x,t)\in\Omega\times[0,T]}\rho(x,t)>\frac{1}{2},	\quad\inf_{(x,t)\in\Omega\times[0,T]}\fT(x,t)>\frac{1}{2},
		\quad \|(\varphi,\theta)(t)\|^2_{L^\infty}\leq \tC\Nhat\ep^2,\\
		& \|\psi(t)\|^2_{L^\infty}\leq \tC\Nhat, \quad
		\|(\nabla\varphi,\nabla\theta)(t)\|^2_{L^\infty}\leq \tC\Nhat,
      \quad \|\nabla\psi(t)\|^2_{L^\infty}\leq \tC\Nhat\ep^{-2},\quad\forall t\in[0,T].  \label{infW}
	\end{aligned}
    \end{equation}
\end{lemma}
\begin{proof}
	Recalling of \eqref{Couette1}, we first obtain from \eqref{chi} that
	\begin{align}
		\|\tfT-1\|_{L^\infty}\leq |\chi-1|+\tC\ep^2\leq \tC\ep+\tC\ep^2\leq \tC\ep.  \notag
	\end{align}
    This yields
     \begin{align}
      \inf_{x\in\Omega}\tfT(x)>\frac{3}{4}, \label{ob1}
     \end{align}
 provided that $\ep$ is small enough. Based on an analogous argument, we derive \eqref{ob0}.
    
    Next, it follows from \eqref{N} that
    \begin{align}
    	\|\theta(t)\|^2_{L^2}\leq \Nhat\ep^{4},\quad 	\|\nabla^2\theta(t)\|^2_{L^2}\leq \Nhat,\quad\forall t\in[0,T],  \notag
    \end{align}
     which, together with \Cref{GN lemma}, leads to
\begin{align}
	\sup_{t\in [0,T]}\|\theta(t)\|^2_{L^\infty}
	\leq \tC\sup_{t\in [0,T]}(\|\theta\|_{L^2}\|\nabla^2\theta\|_{L^2}
	+\|\theta\|^2_{L^2})\leq  \tC \Nhat\ep^2 \leq \frac{1}{4},  \label{ob2}
\end{align}     
provided that $\ep$ is small enough. Thus, \eqref{ob1} and \eqref{ob2} imply
\begin{align}
\inf_{(x,t)\in\Omega\times[0,T]}\fT(x,t)\geq \inf_{x\in\Omega}\tfT(x)-\sup_{t\in [0,T]}\|\theta(t)\|_{L^\infty}>\frac{1}{2}.
\end{align}
Other inequalities in \eqref{infW} can be obtained similarly.

The proof is completed.
\end{proof}

Now, we are in a position to derive the basic energy estimate.
\begin{lemma}   \label{basic}
 Suppose that \eqref{chi} holds. Then, there exist positive constants $\Re_0$ and $\vep_1$ depending only on $\Pr$, $\gamma$ and $\frac{\mup}{\mu}$, such that if $\Re<\Re_0$ and $\ep<\vep_1$, then	
\begin{align}
	&A_0(t)\leq CA_0(0), \quad\quad \forall t\in[0,T]. \label{eneA0}
\end{align}
\end{lemma}
\begin{proof}
We introduce the relative entropy defined by 
\begin{align}
	\eta:=&\frac{\ep^2\rho}{\tfT}|\psi|^2
	+\frac{\rho}{\gamma-1}\big(\frac{\fT}{\tfT}-\ln(\frac{\fT}{\tfT})-1\big)
	+\rho\big(\frac{\tau}{\ttau}-\ln(\frac{\tau}{\ttau})-1\big)   \notag\\
	=&\frac{\ep^2\rho}{\tfT}|\psi|^2
	+\frac{\rho}{\gamma-1}\big(\frac{\theta}{\tfT}-\ln(1+\frac{\theta}{\tfT})\big)
	+\rho\big(-\frac{\varphi}{\rho}-\ln(1-\frac{\varphi}{\rho})\big),  \label{equiv1}
\end{align}
where
\begin{align}
	\tau:=\frac{1}{\rho},\quad \ttau:=\frac{1}{\trho}. \notag
\end{align}
For the function
\begin{align}
	f(z)=z-\ln(1+z),\quad\quad z\in(-1,\infty), \notag
\end{align}
it holds that 
\begin{equation}
	\begin{aligned}
		&f(0)=0,\quad\quad f^\prime(0)=0,\\
		&f^{\prime\prime}(z)=\frac{1}{(1+z)^2}>0,\quad\quad \forall z\in(-1,\infty), \notag
	\end{aligned}
\end{equation}
which yields that there exists a small constant $\delta_0>0$ such that 
\begin{align}
	f(z)\geq z^2, \quad\quad\forall z\in(-\delta_0,\,\delta_0). \notag
\end{align}
Therefore, we deduce form \eqref{infW} and \eqref{equiv1} that
\begin{align}
	\tC^{-1}\|(\varphi,\ep\psi,\theta)(t)\|^2_{L^2}
	\leq \int \eta(t)dx
	\leq \tC\|(\varphi,\ep\psi,\theta)(t)\|^2_{L^2},\quad \forall t\in[0,T], \label{equiv2}
\end{align}	
provided that $\ep$ is small enough.

Differentiating $\eta$ with respect to $t$ gives
\begin{align}
	\pt_t\eta=\left(\ln(\frac{\rho}{\trho})+\frac{\ep^2}{2}\frac{|\psi|^2}{\tfT}
	+\frac{1}{\gamma-1}(\frac{\fT}{\tfT}-\ln(\frac{\fT}{\tfT})-1)\right)\pt_t\varphi
	+\ep^2\frac{\rho}{\tfT}\psi\cdot\pt_t\psi+\rho(\frac{1}{\tfT}-\frac{1}{\fT})\pt_t\theta.   \label{pteta}
\end{align}
It follows from \eqref{eq1} and integration by parts that
\begin{align}
	&\int \left(\ln(\frac{\rho}{\trho})+\frac{\ep^2}{2}\frac{|\psi|^2}{\tfT}
	+\frac{1}{\gamma-1}(\frac{\fT}{\tfT}-\ln(\frac{\fT}{\tfT})-1)\right)\pt_t\varphi dx\notag\\
	=&\int -\left(\ln\rho+\ln\tfT+\frac{\ep^2}{2}\frac{|\psi|^2}{\tfT}
	+\frac{1}{\gamma-1}(\frac{\fT}{\tfT}-\ln\fT+\ln\tfT)\right)\div(\rho u) dx\notag\\
	=&\int u\cdot\nabla\rho dx+\int\rho u\cdot\nabla\left(\frac{\ep^2|\psi|^2}{2\tfT}+\frac{1}{\gamma-1}(\frac{\fT}{\tfT}-\ln\fT)\right)dx
	+\frac{\gamma}{\gamma-1}\int\frac{\rho}{\tfT}\psi\cdot\nabla\tfT dx. \label{eta1}
\end{align}
Similarly, \eqref{eq2}--\eqref{eq3}, together with integration by parts lead to
\begin{align}
	\int \ep^2\frac{\rho}{\tfT}\psi\cdot\pt_t\psi dx
	=&-\ep^2\int\big( \frac{\rho}{\tfT} u\cdot\nabla(\frac{|\psi|^2}{2}) +\frac{\rho}{\tfT}(\psi\cdot\nabla\tu)\cdot\psi\big)dx
	+\int\rho\fT(\frac{\div\psi}{\tfT}-\frac{\psi\cdot\nabla\tfT}{\tfT^2}) dx\notag\\
	&+\ep^2\int\big(\frac{\mu}{\tfT}\Delta\psi\cdot\psi+\frac{(\mu+\mup)}{\tfT}\nabla\div\psi\cdot\psi\big) dx, \label{eta2}
\end{align}
and
\begin{align}
	\int \rho(\frac{1}{\tfT}-\frac{1}{\fT})\pt_t\theta dx
	=&\frac{1}{\gamma-1}\int\rho u\cdot\nabla\big(\ln\fT-\frac{\fT}{\tfT}\big)dx
	+\frac{1}{\gamma-1}\int \frac{\rho\fT}{\tfT^2}u\cdot\nabla\tfT dx\notag\\
	+&\int\big(\rho\div\psi-\frac{\rho\fT}{\tfT}\div\psi\big)dx
 +\int\kappa(\frac{1}{\tfT}-\frac{1}{\fT})\Delta\theta dx  \notag\\
	&+\ep^2\int(\frac{1}{\tfT}-\frac{1}{\fT}) 
	\big(2\mu|\Dpsi|^2+4\mu\Dtu:\Dpsi+\mup(\div\psi)^2\big)dx.\label{eta3}
\end{align}

Substituting \eqref{eq1}--\eqref{eq3} into \eqref{pteta}, and integrating the resulting equation over $\Omega$, we obtain, by using \eqref{eta1}--\eqref{eta3} and some direct algebraic manipulations, that
\begin{align}
	\frac{d}{dt}\int\eta dx
    =&-\frac{\gamma}{\gamma-1}\int\frac{\rho\theta \psi\cdot\nabla\tfT}{\tfT^2} dx  
    		-\ep^2\int  \big(\frac{1}{2}\frac{\rho|\psi|^2\psi\cdot\nabla\tfT}{\tfT^2}
    	+\frac{\rho}{\tfT}(\psi\cdot\nabla\tu)\cdot\psi \big)dx\notag\\
    &-\ep^2\int\big(\frac{\mu}{\fT}|\nabla\psi|^2+\frac{\mu+\mup}{\fT}(\div\psi)^2\big) dx
    -\ep^2\int \frac{\mu(\psi\cdot\nabla\psi)\cdot\nabla\tfT}{\tfT^2}dx\notag\\
     &-\ep^2\int\frac{\mup(\nabla\tfT\cdot\psi)\div\psi}{\tfT^2}dx
    -\int\frac{\kappa}{\fT^2}|\nabla\theta|^2dx
    +\int \frac{\kappa\theta(\tfT+\fT)\nabla\tfT\cdot\nabla\theta}{\tfT^2\fT^2}dx\notag\\
    &+\int \frac{4\mu\ep^2\theta}{\tfT\fT}\Dtu:\Dpsi dx  \notag\\
    :=& \sum_{i=1}^{8}I_i,    \label{I0}
\end{align}
where we have used 
\begin{align}
	2\int |\Dpsi|^2dx=\int\big(|\nabla\psi|^2+(\div\psi)^2\big)dx. \notag
\end{align}
It follows from \eqref{ReMAnew}, \eqref{infW} and \Cref{Poin} that
\begin{align}
I_3+I_6\leq -\frac{c_1}{\Re}\ep^2\|\psi\|^2_{H^1}-\frac{c_2}{\Re}\|\theta\|^2_{H^1},   \notag
\end{align}
where $c_1$ and $c_2$ are positive constants independent of $\ep$, $\Re$ and $T$. Recalling the fact from \eqref{Couette1} and \eqref{chi} that
\begin{align}
	\|\nabla\tfT\|_{L^\infty}=\|(1-\chi)-\frac{\ep^2\Pr }{2C_p}(2x_2-1)\|_{L^\infty}\leq \tC\ep,  \notag
\end{align}
we use \Cref{Poin} and \eqref{ob0} to obtain that
\begin{align}
	I_1\leq \tC\|\nabla\tfT\|_{L^\infty}\|\psi\|_{H^1}\|\theta\|_{H^1}
	\leq \tC\|\nabla\tfT\|^2_{L^\infty}\|\psi\|^2_{H^1}+\tC\|\theta\|^2_{H^1}
	\leq \tC\ep^2\|\psi\|^2_{H^1}+\tC\|\theta\|^2_{H^1}.  \notag
\end{align}
Similarly, we have
\begin{align}
I_2 \leq \tC\|\psi\|^2_{H^1}, 
\quad 	I_4+I_5\leq  \frac{\tC}{\Re}\ep^3\|\psi\|^2_{H^1}, \quad I_7\leq \frac{\tC\ep}{\Re}\|\theta\|^2_{H^1},  \notag
\end{align}
and
\begin{align}
	I_8\leq \frac{\tC}{\Re}\ep^2\|\theta \|_{L^2}\|\nabla\psi\|_{L^2}
	\leq \frac{\tC}{\Re}(\ep^3\|\psi\|^2_{H^1}+\ep\|\theta\|^2_{H^1}), \notag
\end{align}
where we have used the Young inequality. 

Substituting the estimates of $I_1$ through $I_8$ into \eqref{I0}, and integrating the resulting inequality over the time interval $[0,t]$, we get
\begin{align}
	\int \eta(t)dx+C^{-1}\int_{0}^{t}\big(\ep^2\|\psi(s)\|^2_{H^1}+\|\theta(s)\|^2_{H^1}\big)ds
	\leq \int\eta(0)dx,  	\quad\forall t\in[0,T], \label{equiv3}
\end{align}
provided that both $\Re$ and $\ep$ are small enough.

Finally, \eqref{equiv2} and \eqref{equiv3} imply \eqref{eneA0}. The proof is completed.
\end{proof}

\subsection{Estimates for the first order derivatives.}
This subsection is devoted to establishing the \textit{a priori} $H^1$-type estimates. 

We first derive the estimates for $\sup_{t\in [0,T]}A_1(t)$ and $\int_{0}^{T}\|\pt_t \varphi(t)\|^2_{L^2}dt$.
\begin{lemma} \label{lemma1}
Suppose that \eqref{chi} holds, and suppose that $\Re<\Re_0$ and $\ep<\vep_1$ as in \Cref{basic}. Then, the following estimates hold:
	\begin{align}
		&\int_{0}^{t}\|\pt_t \varphi(s)\|^2_{L^2}ds
		\leq C\int_{0}^{t}\big(\|\psi(s)\|^2_{H^1}
		+\|\pt_{1}\varphi(s)\|^2_{L^2}\big)ds,
		\quad\forall t\in[0,T],    \label{ene1}	
\end{align}
	and
     \begin{align}
    	&A_1(t)\leq CA_1(0)+C\ep^{-2}A_0(0)+C\ep^{-2}\int_{0}^{t}\|\pt_1\varphi\|^2_{L^2}ds,
    \quad\forall t\in[0,T].\label{ene2} 
    \end{align}
\end{lemma}

\begin{proof}
	It follows from \eqref{eq1}, \eqref{infW} and \Cref{Poin} that
	\begin{align}
   	\|\pt_t\varphi\|_{L^2}\leq& \|\tu\cdot\nabla\varphi\|_{L^2}
   	+(\|\nabla\varphi\|_{L^\infty}+\|\nabla\trho\|_{L^\infty})\|\psi\|_{L^2}+\|\rho\|_{L^\infty}\|\nabla\psi\|_{L^2}\notag\\
    \leq&\|\pt_{1}\varphi\|_{L^2}+ C\|\psi\|_{H^1},    \label{pt1}
	\end{align}
    where the fact that $\tu^2=0$ has been used. Integrating \eqref{pt1} over the time interval $[0,t]$ gives \eqref{ene1}.
	
	Next, multiplying \eqref{eq3} by $\pt_t \theta$, we obtain, by using the method of integration by parts, that
	\begin{align}
	&\frac{d}{dt}(\frac{\kappa}{2}\|\nabla\theta\|^2_{L^2})+\int\frac{\rho}{\gamma-1} (\pt_t\theta)^2dx\notag\\
	\leq&\delta\|\pt_t\theta\|^2_{L^2}
	+C(1+\frac{1}{\delta})\big(\|\psi\|^2_{H^1}+\|\theta\|^2_{H^1}
	+\ep^2(1+\|\nabla\psi\|^2_{L^\infty})\|\nabla\psi\|^2_{L^2}\big)\notag\\
	\leq &\delta\|\pt_t\theta\|^2_{L^2}+C(1+\frac{1}{\delta})\big(\|\psi\|^2_{H^1}+\|\theta\|^2_{H^1}\big),
	\quad\quad \forall\,\delta>0, \label{pt2}
	\end{align}
    where we have used the Young inequality in the first inequality , and we have used \eqref{infW} in the second inequality. Integrating \eqref{pt2} over the time interval $[0,t]$ and choosing $\delta$ suitably small, we get
    \begin{align}
    &\|\nabla \theta(t)\|^2_{L^2}+\int_{0}^{t}\|\pt_t \theta(s)\|^2_{L^2}ds \notag\\
    \leq& C\|\nabla \theta_0\|^2_{L^2}
    +C\int_{0}^{t}\big(\|\psi(s)\|^2_{H^1}+\|\theta(s)\|^2_{H^1}\big)ds,\quad\forall t\in[0,T].   \label{tempene}
   	\end{align}
    Similarly, multiplying $\pt_t \psi$ on \eqref{eq2}, we obtain, by using the method of integration by parts, that
    \begin{align}
    	&\frac{d}{dt}\big(\frac{\mu}{2}\|\nabla\psi\|^2_{L^2}+\frac{\mu+\mup}{2}\|\div\psi\|^2_{L^2}\big)
    	+\int\rho (\pt_t\psi)^2dx\notag\\
    	\leq & \delta\|\pt_t\psi\|^2_{L^2}+C(1+\frac{1}{\delta})\big(\|\psi\|^2_{H^1}+\ep^{-2}\int\nabla P(\rho,\theta)\pt_t\psi dx\big), \quad \forall\,\delta>0.     \label{pt3}
    \end{align}
    where we have used \eqref{infW} and the Young inequality.
   Note that
      \begin{align}
      	 &\ep^{-2}\int\nabla P(\rho,\theta)\pt_t\psi dx\notag\\
      	 =&\ep^{-2}\int\nabla \big(P(\rho,\theta)-P(\trho,\tthe)\big)\pt_t\psi dx
      	 =-\ep^{-2}\int (\rho\theta-\trho\tthe)\pt_t\div\psi dx\notag\\
       =&-\ep^{-2}\pt_t\big(\int(\rho\theta-\trho\tthe)\div\psi dx\big)
       +\ep^{-2}\int\pt_t(\rho\theta)\div\psi dx\notag\\
       =&-\ep^{-2}\pt_t\big(\int(\theta\varphi+\trho\theta)\div\psi dx\big)
       +\ep^{-2}\int(\varphi\pt_t\theta+\theta\pt_t\varphi+\trho\pt_t\theta)\div\psi dx\notag\\
       \leq&-\ep^{-2}\pt_t\big(\int(\theta\varphi+\trho\theta)\div\psi dx\big)
       +C\ep^{-2}\big(\|\pt_t\varphi\|^2_{L^2}
       +\|\pt_t\theta\|^2_{L^2}+\|\psi\|^2_{H^1}\big).      \label{pt4}
      \end{align}
   The Young inequality gives 
   \begin{align}
   	\ep^{-2}\int(\theta\varphi+\trho\theta)\div\psi dx
   	\leq \delta\|\nabla\psi\|^2_{L^2}
   	+C(1+\frac{1}{\delta})\ep^{-4}\big(\|\varphi\|^2_{L^2}+\|\theta\|^2_{L^2}\big),
   	\quad\forall \delta>0.         \label{pt5}
   \end{align}
 and
\begin{align}
     \ep^{-2}\int(\theta_0\varphi_0+\trho\theta_0)\div\psi_0 dx
	\leq C\|\nabla\psi_0\|^2_{L^2}
	+C\ep^{-4}\big(\|\varphi_0\|^2_{L^2}+\|\theta_0\|^2_{L^2}\big),
	\,\,\,\,\,\forall \delta>0.         \label{pt6}
\end{align}
  Therefore, based on \eqref{pt4}--\eqref{pt6}, integrating \eqref{pt3} over the time interval $[0,t]$, and choosing $\delta$ small enough, we have
  \begin{align}
  	&\|\nabla \psi(t)\|^2_{L^2}+\int_{0}^{t}\|\pt_t \psi(s)\|^2_{L^2}ds \notag\\
  	\leq &C\ep^{-4}\big(\|(\varphi,\theta)(t)\|^2_{L^2}
  	+\|(\varphi_0,\theta_0)\|^2_{L^2}\big)+C\|\nabla \psi_0\|^2_{L^2}\notag\\
  	&+C\ep^{-2}\int_{0}^{t}\big(\|\pt_t\varphi\|^2_{L^2}
  	+\|\pt_t\theta\|^2_{L^2}+\|\psi\|^2_{H^1}\big)ds,
  	\quad\forall t\in[0,T].\label{tempene3} 
  \end{align}
Finally, adding \eqref{tempene} multiplied by $2C\ep^{-2}$ to \eqref{tempene3} derives
  \begin{align}
	A_1(t)\leq& CA_1(0)+C\ep^{-2}A_0(t)+C\ep^{-2}\int_{0}^{t}\|\pt_t\varphi\|^2_{L^2}ds\notag\\
	\leq& CA_1(0)+C\ep^{-2}A_0(0)+C\ep^{-2}\int_{0}^{t}\|\pt_1\varphi\|^2_{L^2}ds,
	\quad\forall t\in[0,T],
\end{align}
where \eqref{eneA0} and \eqref{ene1} have been used in the second inequality.

The proof is completed.
\end{proof}
    
Next, the equations \eqref{eq1} and \eqref{eq2} can be rewritten as
 \begin{align}
 	&\pt_t\varphi+u\cdot\nabla\varphi+\div\psi=f_1,   \label{eqf1}
 	 \end{align}
  and
   \begin{align}
 	&\rho\big(\pt_t\psi+u\cdot\nabla\psi+\psi\cdot\nabla\tu\big)
 	+\ep^{-2}\big(\nabla\varphi+\nabla\theta\big)
 	-\mu\Delta\psi-(\mu+\mup)\nabla\div\psi=\ep^{-2}\nabla f_2,\label{eqf2}
 \end{align}
 where 
 \begin{align}
 	f_1=&-(\rho-1)\div\psi-\psi\cdot\nabla\trho, \quad\quad
 	f_2=-\big((\trho-1)\theta+(\tfT-1)\varphi+\varphi\theta\big). \notag 
 \end{align}
We derive the estimates for $\sup_{t\in [0,T]}\ep^{-2}\|\nabla\varphi(t)\|^2_{L^2}$
and $\int_{0}^{T}\|\nabla\div \psi(t)\|^2_{L^2}dt$ as follows.  

\begin{lemma}  \label{lemma2}
Suppose that \eqref{chi} holds, and suppose that $\Re<\Re_0$ and $\ep<\vep_1$ as in \Cref{basic}. Then, the following estimate holds:
	\begin{align}
		 &\ep^{-2}\|\nabla\varphi(t)\|^2_{L^2}+\|\pt_{1}\psi(t)\|^2_{L^2}+\int_{0}^{t}\|\nabla\div \psi(s)\|^2_{L^2}ds\notag\\
		\leq& C\big(\ep^{-2}\|\varphi_0\|^2_{H^1}+\|\pt_{1}\psi_0\|^2_{L^2}
		+ A_1(0)+\ep^{-2}A_0(0)
	+\ep^{-2}\int_{0}^{t}\|\nabla\varphi(s)\|^2_{L^2}ds\big),\quad \forall t\in[0,T]. \label{ene3}
	\end{align}
	
\end{lemma}
\begin{proof}
	Noting that $\pt_{1}\trho=0$ and $\pt_{1}\tu=0$, we first differentiating both \eqref{eq1} and \eqref{eq2} in $x_1$ to get
	\begin{align}
	\pt_t\pt_{1}\varphi+u\cdot\nabla\pt_{1}\varphi+\div(\pt_{1}\psi)
	=\pt_{1}f_1-\pt_{1}\psi\cdot\nabla\varphi, \label{pertur6}
	\end{align}
and	
	\begin{align}
	&\rho(\pt_t\pt_{1}\psi+u\cdot\nabla \pt_{1}\psi)
	+\ep^{-2}\nabla \big(\pt_{1}\varphi+\pt_{1}\theta)-\mu\Delta \pt_{1}\psi-(\mu+\mup)\nabla\pt_{1}\div \psi\notag\\
	=&\nabla\pt_1f_2+R_1+R_2,\label{pertur7}
	\end{align}
where
\begin{align}
\quad R_1=-\pt_{1}\varphi(\pt_t\psi+u\cdot\nabla\psi+\psi\cdot\nabla\tu),\quad R_2=-\rho\pt_{1}\psi\cdot\nabla(\tu+\psi). \notag
\end{align}
It follows from \eqref{infW} and \eqref{N} that
\begin{align}
	R_1=O(1)|\pt_1\varphi|(|\pt_t\psi|+|\nabla\psi|+|\psi|),
	\quad R_2=O(1)|\nabla\psi|(1+|\nabla\psi|).  \label{Landou1}
\end{align}
Throughout this paper, the Landau notation $O(1)$ is used to indicate a function whose absolute value remains uniformly bounded by a positive constant $C$ independent of $\ep$ and $T$.

Similarly, we have 
\begin{align}
	&\pt_{1}f_1=O(1)(|\trho-1|+|\varphi|)|\pt_{1}\div\psi|+O(1)|\pt_1\varphi||\div\psi|+O(1)|\nabla\trho||\pt_{1}\psi|,
	\label{Landou2}
\end{align}
and 
\begin{align}
		\pt_{1}f_2=O(1)(|\trho-1|+|\tfT-1|+|\varphi|+|\theta|)(|\pt_1\varphi|+|\pt_1\theta|).\label{Landou3}
\end{align}
    
    Then, based on \eqref{ob0} and \eqref{infW}, adding \eqref{pertur6} multiplied by $\ep^{-2}\pt_{1}\varphi$ to \eqref{pertur7} multiplied by $\pt_{1}\psi$, we obtain, by applying the method of integration by parts, and using \eqref{Landou1} and \eqref{Landou2}--\eqref{Landou3} that
    \begin{align}
    	&\frac{d}{dt}\big(\frac{1}{2}\int_{\Omega}\ep^{-2}|\pt_{1}\varphi|^2+\rho|\pt_{1}\psi|^2dx\big)
    	+\mu\|\nabla\pt_{1}\psi\|^2_{L^2}+(\mu+\mup)\|\pt_{1}\div\psi\|^2_{L^2}    \notag\\
    	\leq&  \frac{1}{2} \ep^{-2}\|\pt_{1}\varphi\|_{L^\infty}\|\div\psi\|_{L^2}\|\pt_{1}\varphi\|_{L^2}
    	+\ep^{-2}\|\pt_1f_1\|_{L^2}\|\pt_1\varphi\|_{L^2}
    	+\ep^{-2}\|\nabla\varphi\|_{L^\infty}\|\pt_1\psi\|_{L^2}\|\pt_1\varphi\|_{L^2} \notag\\
    	&+\ep^{-2}\big(\|\pt_1\theta\|_{L^2}+\|\pt_1f_2\|_{L^2}\big)\|\div\pt_1\psi\|_{L^2}
    	+\big(\|R_2\|_{L^2}+\|R_3\|_{L^2}\big)\|\pt_1\psi\|_{L^2}
        \notag\\
    	\leq& \frac{\mu}{4}\|\nabla\pt_{1}\psi\|^2_{L^2}+C(1+\frac{1}{\delta})\ep^{-2}\|\pt_1\varphi\|^2_{L^2}
    	+\delta\ep^{-2}\|\pt_1f_1\|^2_{L^2}	\notag\\
    	&+C\big(\ep^{-4}\|\pt_1\theta\|^2_{L^2}+\ep^{-4}\|\pt_1f_2\|^2_{L^2}
    	+\|R_2\|^2_{L^2}+\|R_3\|^2_{L^2}+\ep^{-2}\|\nabla\psi\|^2_{L^2}\big)
    	\notag\\
    	\leq&  (\frac{\mu}{4}+C\delta)\|\nabla\pt_{1}\psi\|^2_{L^2}
    	+C(1+\frac{1}{\delta})\ep^{-2}\|\pt_1\varphi\|^2_{L^2}\notag\\
    		&+C\big(\ep^{-4}\|\theta\|^2_{H^1}
    	+\|\pt_t\psi\|^2_{L^2}+\ep^{-2}\|\psi\|^2_{H^1}\big),    \quad \quad \forall \delta>0,   \label{div1}
    \end{align}
   where we have used \Cref{Poin} and the Young inequality.
   
Integrating \eqref{div1} over the time interval $[0,t]$ and choosing $\delta$ small enough such that $C\delta\leq \frac{\mu}{4}$, we have
 \begin{align}
	&\ep^{-2}\|\pt_{1}\varphi(t)\|^2_{L^2}+\|\pt_{1}\psi(t)\|^2_{L^2}
	+\int_{0}^{t}\big(\|\nabla\pt_{1}\psi(s)\|^2_{L^2}+\|\pt_{1}\div\psi(s)\|^2_{L^2}\big)ds
     \notag\\
	\leq& C\big(\ep^{-2}\|\pt_{1}\varphi_0\|^2_{L^2}+\|\pt_{1}\psi_0\|^2_{L^2}\big)
	+C\ep^{-2}\int_{0}^{t} \|\pt_{1}\varphi(s)\|^2_{L^2}ds
	\notag\\
	&+C\int_{0}^{t}\big(\ep^{-4}\|\theta(s)\|^2_{H^1}+\ep^{-2}\|\psi(s)\|^2_{H^1}
	+\|\pt_t\psi(s)\|^2_{L^2}\big)ds.    \label{div2}
\end{align}
   
   	Next, noting that $	\pt_{2}\tu\cdot\nabla\varphi=\pt_{2}\tu^1\pt_{1}\varphi=\pt_{1}\varphi$, we obtain by differentiating \eqref{eq1} in $x_2$ that
   \begin{align}
   	&\pt_t\pt_{2}\varphi+u\cdot\nabla\pt_{2}\varphi+\pt_{1}\varphi
   	+\rho\pt^2_{2}\psi^2 = R_3,   \label{eqnormal1}
   \end{align}
   where 
      \begin{align}
      R_3=&-\pt_2\psi\cdot\nabla\rho-\psi\cdot\nabla\pt_2\trho-\pt_2\rho\div\psi-\rho\pt_{2}\pt_{1}\psi^1\notag\\
      =&O(1)\big(|\nabla \psi|(|\nabla \trho|+|\nabla \varphi|)+|\psi||\nabla^2 \trho|+|\rho||\nabla\pt_{1}\psi|\big).   \notag
      \end{align}
We note the fact from direct calculation that
\begin{align}
       \mu\Delta \psi^2+(\mu+\mup)\pt_{2}\div \psi
       =(2\mu+\mup)\pt^2_{2}\psi^2+\mu\pt_{1}(\pt_{1}\psi^2-\pt_{2}\psi^1)
         +(\mu+\mup)\pt_{2}\pt_{1}\psi^1,\label{normal1}
\end{align}
and
\begin{align}
	\pt_{2} P(\rho,\fT)=\pt_{2}\big(P(\rho,\fT)-P(\trho,\tfT)\big)              
	=H_1+H_2,  \label{normal2}
\end{align}
where
\begin{align}
	H_1:=\tfT\pt_{2}\varphi+\varphi\pt_{2}\tfT,\quad
	H_2:=\theta\pt_{2}\varphi+\varphi\pt_{2}\theta+\theta\pt_{2}\trho+\trho\pt_{2}\theta.    \label{H1}
\end{align}
Thus, we derive from \eqref{eq2}, \eqref{normal1} and \eqref{normal2} that
   \begin{align}
       &-(2\mu+\mup)\pt^2_{2}\psi^2
       +\ep^{-2}H_1=R_4,\label{eqnormal2}
  \end{align}
where
 \begin{align}
	R_4=&\mu\pt_{1}(\pt_{1}\psi^2-\pt_{2}\psi^1)+(\mu+\mup)\pt_{2}\pt_{1}\psi^1
	-\rho\big(\pt_t\psi+u\cdot\nabla\psi+\psi\cdot\nabla\tu\big)-\ep^{-2}H_2   \notag\\
	=&O(1)(|\nabla\pt_{1}\psi|+|\pt_t\psi|+|\nabla\psi|+|\psi|+\ep^{-2}(|\theta|+|\nabla\theta|)). \notag
\end{align}

    Adding \eqref{eqnormal1} multiplied by $(2\mu+\mup)\ep^{-2}H_1$ to \eqref{eqnormal2} multiplied by $\ep^{-2}\rho H_1$, we obtain, by using the method of integration by parts, that
   \begin{align}
   	&(2\mu+\mup)\ep^{-2}\frac{d}{dt}\int\big(\frac{1}{2}\tthe|\pt_{2}\varphi|^2
   	+\varphi\pt_{2}\tthe\pt_{2}\varphi\big) dx
   +\int_{\Omega}\ep^{-4}\rho H_1^2dx  \notag\\
   \leq &(2\mu+\mup)\ep^{-2}\int \big(|\div(\tthe\psi)||\pt_{2}\varphi|^2
   +|\pt_{2}\tthe||\pt_{2}\varphi||\pt_{t}\varphi|+|H_1||\pt_{1}\varphi|\big)dx\notag\\
   &+C\ep^{-2}(\| R_3\|_{L^2}+\| R_4\|_{L^2})\| H_1\|_{L^2}
   \notag\\
   \leq & \frac{1}{2}\big(\inf_{(x,t)\in \Omega\times[0,T]}\rho(x,t)\big)
   \int\ep^{-4}H_1^2dx+C\big(\|\pt_t\psi\|^2_{L^2}+\|\nabla\psi\|^2_{L^2}
   +\ep^{-4}\|\nabla\theta\|^2_{L^2}\notag\\
   &+\|\nabla\pt_{1}\psi\|^2_{L^2}
   +\ep^{-2}\|\pt_{2}\varphi\|^2_{L^2}
   +\ep^{-2}\|\pt_t\varphi\|^2_{L^2}
   +\|\pt_{1}\varphi\|^2_{L^2}\big), \label{div3}
   \end{align}
   where we have used \Cref{Poin} and the Young inequality.
   
Integrating \eqref{div3} over the time interval $[0,t]$, we have
\begin{align}
	&\ep^{-2}\|\pt_{2}\varphi(t)\|^2_{L^2}
	+\int_{0}^{t}\ep^{-4}\|H_1(s)\|^2_{L^2}ds\notag\\
	\leq &C\ep^{-2}\big(\|\varphi_0\|^2_{H^1}+\|\varphi(t)\|^2_{L^2}\big)
	+C\int_{0}^{t}\big(\|\pt_t\psi(s)\|^2_{L^2}+\|\psi(s)\|^2_{H^1}
	+\ep^{-4}\|\theta(s)\|^2_{H^1}\notag\\
	&+\|\nabla\pt_{1}\psi(s)\|^2_{L^2}+\ep^{-2}\|\pt_{2}\varphi(s)\|^2_{L^2}
	+\ep^{-2}\|\pt_t\varphi(s)\|^2_{L^2}
	+\|\pt_{1}\varphi(s)\|^2_{L^2}\big)ds, \quad\forall t\in[0,T],  \label{div4}
\end{align}
where we have used \eqref{infW} and the fact from the Young inequality that
\begin{align}
	\ep^{-2}\int_{\Omega}\tthe|\pt_{2}\varphi|^2+\varphi\pt_{2}\tthe\pt_{2}\varphi dx
	\geq& \ep^{-2}\int_{\Omega}\frac{1}{2}\tthe|\pt_{2}\varphi|^2dx
	-C\ep^{-2}\|\frac{\pt_{2}\tthe}{\tthe}\|_{L^\infty}\|\varphi\|^2_{L^2}\notag\\
	\geq& C\ep^{-2}\|\pt_{2}\varphi\|^2_{L^2}-C\ep^{-2}\|\varphi\|^2_{L^2}.\notag
\end{align}
It follows from \eqref{eqnormal2} that
\begin{align}
	(2\mu+\mup)\pt_{2}\div\psi
	=\ep^{-2}H_1+O(1)\big(|\pt_t\psi|+|\nabla\psi|+|\psi|
	+\ep^{-2}(|\theta|+|\nabla\theta|)+|\nabla\pt_{1}\psi|\big).\label{eqnormal3}
\end{align}
This, together with \eqref{div4}, gives
\begin{align}
	&\ep^{-2}\|\pt_{2}\varphi(t)\|^2_{L^2}
	+\int_{0}^{t}\|\pt_{2}\div\psi(s)\|^2_{L^2}ds\notag\\
	\leq& \int_{0}^{t}\ep^{-4}\|H_1(s)\|^2_{L^2}ds
	+C\int_{0}^{t}\big(\|\pt_t\psi(s)\|^2_{L^2}+\|\nabla\psi(s)\|^2_{L^2}
	+\|\nabla\theta(s)\|^2_{L^2}+\|\nabla\pt_{1}\psi(s)\|^2_{L^2}\big)ds\notag\\
	\leq  &C\ep^{-2}\|\varphi_0\|^2_{H^1}+C_1\int_{0}^{t}\|\nabla\pt_{1}\psi(s)\|^2_{L^2}ds
	+C\int_{0}^{t}\big(\|\pt_t\psi(s)\|^2_{L^2}+\|\psi(s)\|^2_{H^1}
	+\ep^{-4}\|\theta(s)\|^2_{H^1}\notag\\
	&+\ep^{-2}\|\pt_{2}\varphi(s)\|^2_{L^2}
	+\ep^{-2}\|\pt_t\varphi(s)\|^2_{L^2}
	+\|\pt_{1}\varphi(s)\|^2_{L^2}\big)ds,
	\quad\quad\forall t\in[0,T], \label{div5}
\end{align}
where $C_1$ is a positive constant independent of $\ep$ and $T$.

Finally, multiplying \eqref{div2} by $2C_1$ and then adding the resulting inequality to \eqref{div5}, we have
\begin{align}
	&\ep^{-2}\|\nabla\varphi(t)\|^2_{L^2}+\|\pt_{1}\psi(t)\|^2_{L^2}
	+\int_{0}^{t}\big(\|\nabla\pt_{1}\psi(s)\|^2_{L^2}+\|\nabla\div \psi(s)\|^2_{L^2}\big)ds\notag\\
	\leq& C\big(\ep^{-2}\|\varphi_0\|^2_{H^1}+\|\pt_{1}\psi_0\|^2_{L^2}\big)
	+C\int_{0}^{t}\big(\ep^{-2}\|\pt_{1}\varphi(s)\|^2_{L^2}
	+\ep^{-4}\|\theta(s)\|^2_{H^1}+\ep^{-2}\|\psi(s)\|^2_{H^1}\notag\\
	&+\|\pt_t\psi(s)\|^2_{L^2}
	+\ep^{-2}\|\pt_{2}\varphi(s)\|^2_{L^2}
	+\ep^{-2}\|\pt_t\varphi(s)\|^2_{L^2}\big)ds,  	\quad\quad  \forall\, t\in[0,T], \notag
\end{align}
which, together with \eqref{ene1} and \eqref{ene2}, implies \eqref{ene3}.

The proof is completed.
\end{proof}

To derive the estimates for $\int_{0}^{T}\|\nabla^2\psi(t)\|^2_{L^2}dt$ and $\int_{0}^{T}\ep^{-2}\|\varphi(t)\|^2_{H^1}dt$, we rewrite \eqref{eq2} in the form of stationary nonhomogeneous Stokes equations:
\begin{equation}
	\left.\begin{cases}
		\div \psi=g, \quad\quad\quad\quad\quad\text{in}\,\,\Omega,\\
		-\mu\Delta \psi+\ep^{-2}\nabla P(\rho,\fT)=F,
		\quad\quad\quad\text{in}\,\,\Omega,\\
		\psi|_{x_2=0,1}=0,
	\end{cases}\label{Stokes2}
	\right.
\end{equation}
where 
\begin{align}
	g=\div\psi,\,\,\, F=-\rho\big(\pt_t\psi+u\cdot\nabla \psi+\psi\cdot\nabla\tu\big)+(\mu+\mup)\nabla\div\psi. \label{gF}
\end{align}
We have the following lemma.
\begin{lemma}  \label{lemma3}
	Suppose that \eqref{chi} holds, and suppose that $\Re<\Re_0$ and $\ep<\vep_1$ as in \Cref{basic}. Then, there exists a positive constant $\vep_2$ depending only on $\Re$, $\Pr$, $\frac{\mup}{\mu}$ and $\gamma$, such that if $\ep\leq \vep_2$, then
	\begin{align}
	\int_{0}^{t}\|\psi(s)\|^2_{H^2}ds+\ep^{-4}\int_{0}^{t}\|\varphi(s)\|^2_{H^1}ds
	\leq  C\big(A_1(0)+\ep^{-2}A_0(0)\big), \quad \forall t\in [0,T].
		\label{ene4}
	\end{align}
\end{lemma}
\begin{proof}
We first obtain from \Cref{Stokes lemma} that
\begin{align}
	\mu\|\psi\|^2_{H^2}+\ep^{-4}\|\nabla P(\rho,\fT)\|^2_{L^2}
	\leq& C(\|F\|^2_{L^2}+\|g\|^2_{L^2})\notag\\
	\leq& C\big(\|\nabla\div\psi\|^2_{L^2}+\|\psi\|^2_{H^1}+\|\pt_t\psi\|^2_{L^2}\big), \label{Stokes3}
\end{align} 
where we have used \eqref{N} and the \Cref{Poin}.

Next, we observe that 
\begin{align}
	\pt_{1}P(\rho,\fT)=\fT\pt_{1}\rho+\rho\pt_{1}\fT
	=\fT\pt_{1}\varphi+\rho\pt_{1}\theta,     \notag
\end{align}
which, together with \eqref{infW}, gives that
\begin{align}
	&\ep^{-4}\int_{0}^{t}\|\pt_{1}\varphi(s)\|^2_{L^2}ds 
	\leq C\ep^{-4}\int_{0}^{t}\big(\|\pt_{1} P(\rho,\fT)(s)\|^2_{L^2}  
	+\|\pt_{1} \theta(s)\|^2_{L^2}\big)ds. \label{ptx1rho} 
\end{align}	
Similarly, it follows from \eqref{normal2}, \eqref{H1} and \eqref{infW} that
\begin{align}
	\ep^{-4}\int_{0}^{t}\|H_1(s)\|^2_{L^2}ds 
	\leq& 2\ep^{-4}\int_{0}^{t}\big(\|\pt_{2} P(\rho,\fT)(s)\|^2_{L^2}  
	+\|H_2(s)\|^2_{L^2}\big)ds \notag\\
	\leq& 2\ep^{-4}\int_{0}^{t}\|\pt_{2} P(\rho,\fT)(s)\|^2_{L^2} ds 
	+C\ep^{-4}\int_{0}^{t}\|\theta(s)\|^2_{H^1}ds,  \label{combine1} 
\end{align}
and
\begin{align}
	\ep^{-4}\int_{0}^{t}\|\pt_2\varphi(s)\|^2_{L^2}ds
	\leq& C\ep^{-4}\int_{0}^{t}\|\tfT\pt_2\varphi(s)\|^2_{L^2}ds\notag\\
	\leq& C\ep^{-4}\int_{0}^{t}\|H_1(s)\|^2_{L^2}ds
	+C\ep^{-4}\int_{0}^{t}\|\nabla\tfT\|^2_{L^\infty}\|\varphi(s)\|^2_{L^2}ds\notag\\
	\leq& C\ep^{-4}\int_{0}^{t}\|H_1(s)\|^2_{L^2}ds+C\ep^{-2}\int_{0}^{t}\|\varphi(s)\|^2_{L^2}ds.\label{combine2} 
\end{align}
Combining \eqref{combine1} and \eqref{combine2} derives that
\begin{align}
	\ep^{-4}\int_{0}^{t}\|\pt_2\varphi(s)\|^2_{L^2}ds
	\leq&  C\ep^{-4}\int_{0}^{t}\big(\|\nabla P(\rho,\fT)(s)\|^2_{L^2}+\|\theta(s)\|^2_{H^1}\big)ds\notag\\
	&+C\ep^{-2}\int_{0}^{t}\|\varphi(s)\|^2_{L^2}ds. \label{ptx2rho}
\end{align}

Finally, it follows from  \eqref{Stokes3}, \eqref{ptx1rho}, \eqref{ptx2rho} and \Cref{Poin} that
\begin{align}
	&\int_{0}^{t}\|\psi(s)\|^2_{H^2}ds+\ep^{-4}\int_{0}^{t}\|\nabla\varphi(s)\|^2_{L^2}ds\notag\\
    \leq& \int_{0}^{t}\|\psi(s)\|^2_{H^2}ds+ C\ep^{-4}\int_{0}^{t}\big(\|\nabla P(\rho,\fT)(s)\|^2_{L^2}+\|\theta(s)\|^2_{H^1}\big)ds
    +C\ep^{-2}\int_{0}^{t}\|\varphi(s)\|^2_{L^2}ds\notag\\
	\leq&  C\int_{0}^{t}\big(\|\nabla\div\psi\|^2_{L^2}+\|\psi\|^2_{H^1}+\|\pt_t\psi\|^2_{L^2} 
	+\ep^{-4}\|\theta(s)\|^2_{H^1}\big)ds+C\ep^{-2}\int_{0}^{t}\|\nabla\varphi(s)\|^2_{L^2}ds,\notag
\end{align}
which, together with \eqref{ene1}, \eqref{ene2}, \eqref{ene3} and \eqref{eneA0}, leads to 
\begin{align}
		&\int_{0}^{t}\|\psi(s)\|^2_{H^2}ds+\ep^{-4}\int_{0}^{t}\|\nabla\varphi(s)\|^2_{L^2}ds\notag\\
		\leq&  CA_1(0)+C\ep^{-2}A_0(0)+C\ep^{-2}\int_{0}^{t}\|\nabla\varphi(s)\|^2_{L^2}ds. \notag
\end{align}
Therefore, with the help of \Cref{Poin}, there exists a positive constant $\vep_2$ depending only on $\Re$, $\Pr$, $\frac{\mup}{\mu}$ and $\gamma$, such that if $\ep\leq \vep_2$, then \eqref{ene4} holds.

The proof is completed.
\end{proof}

To obtain the estimate for $\int_{0}^{T}\ep^{-2}\| \nabla^2\theta\|^2_{L^2}dt$, we note that $\theta$ satisfies the elliptic equation
\begin{equation}
	\left.\begin{cases}
		\begin{aligned}
	-\kappa\Delta \theta=&-\rho\big(\pt_t\theta+ u\cdot\nabla\theta+\psi\cdot\nabla\tfT\big)
	+P(\rho,\fT) \div\psi  \\
	&+ \ep^{2}\big(2\mu|\Dpsi|^2+\mup(\div\psi)^2+4\mu \Dtu:\Dpsi\big)
	\quad\quad\text{in}\,\,\Omega, 
	\end{aligned}\\
		\theta|_{x_2=0,1}=0.
	\end{cases}\label{ellptictheta}
	\right.
\end{equation}
Then, we have the following lemma.
\begin{lemma} \label{lemma4}
	Suppose that \eqref{chi} holds, and suppose that $\Re<\Re_0$ and $\ep\leq\vep_2$ as in Lemma \ref{lemma3}. Then, we have
\begin{align}
	\ep^{-2}\int_{0}^{t}\|\nabla^2 \theta(s)\|^2_{L^2}ds
	\leq C\big(A_1(0)+\ep^{-2}A_0(0)\big),\quad\quad\quad \forall t\in[0,T].  	\label{ene5}
\end{align}
\end{lemma}
\begin{proof}
	Based on the classical elliptic theory \cite{ADN1959,ADN1964}, we obtain from \eqref{ellptictheta}, \eqref{infW} and \eqref{ene2} that
	\begin{align}
		&\ep^{-2}\int_{0}^{t}\|\nabla^2 \theta(s)\|^2_{L^2}ds\notag\\
		\leq& C\ep^{-2}\int_{0}^{t}\big(\|\pt_t \theta(s)\|^2_{L^2}+\|\nabla\theta(s)\|^2_{L^2}
		+\|\nabla\tfT\|^2_{L^\infty}\|\psi(s)\|^2_{L^2}
		+\ep^2(1+\|\nabla\psi(s)\|^2_{L^\infty})\|\nabla \psi(s)\|^2_{L^2}\big)ds\notag\\
		\leq& C\ep^{-2}\int_{0}^{t}\big(\|\pt_t \theta(s)\|^2_{L^2}+\|\theta(s)\|^2_{H^1}
		+\|\psi(s)\|^2_{H^1}\big)ds\notag\\
		\leq& CA_1(0)+C\ep^{-2}A_0(0),  
		\quad\quad\quad \forall t\in[0,T].  \notag
	\end{align}

The proof is completed.
\end{proof}

In summary of Lemma \ref{lemma1}--\ref{lemma4}, we have the following lemma.
\begin{lemma}\label{lemmaH1}
		Suppose that \eqref{chi} holds, and suppose that $\Re<\Re_0$ and $\ep\leq\vep_2$ as in Lemma \ref{lemma3}. Then, we have
	\begin{align}
    &A_1(t)+A_2(t) \leq C\big(A_1(0)+A_2(0)+\ep^{-2}A_0(0)\big), 
    \quad\quad \forall t\in[0,T].  \label{eneA12}
	\end{align}
\end{lemma}
\begin{proof}
	\eqref{eneA12} follows from \eqref{ene1}, \eqref{ene3}, \eqref{ene4} and \eqref{ene5}. The proof is completed.
\end{proof}

\subsection{Estimates for high order derivative.}
This subsection is devoted to the high order \textit{a priori} estimates, which can be derived by using the similar argument as in the previous subsection or \cite{MN1983}. Comparing to \cite{MN1983} and the argument in the previous subsection, the major difference are the terms that come from the products of lower order derivatives of $\tW$ and $w$ in \eqref{eq1}--\eqref{eq3}. Such terms can be dealt with by using the fact that $\div\tu=0$ and the \textit{a priori} estimates for lower order derivatives obtained in \eqref{eneA0} and \eqref{eneA12}. Therefore, we merely sketch the proof.

\begin{lemma} \label{lemmaH2temp}
	Suppose that \eqref{chi} holds, and suppose that $\Re<\Re_0$ and $\ep\leq\vep_2$ as in Lemma \ref{lemma3}. Then, 
		\begin{align}
		\ep^{2}A_3(t)\leq C\big(\ep^{2}A_3(0)+A_1(0)+A_2(0)+\ep^{-2}A_0(0)\big),
		\quad  \forall t\in[0,T],\label{eneA3}
	\end{align}
and 
    	\begin{align}
    	\ep^2\|\nabla^2\psi(t)\|^2_{L^2}+\|\nabla^2\theta(t)\|^2_{L^2}
    	\leq C\big(\ep^{2}A_3(0)+A_1(0)+A_2(0)+\ep^{-2}A_0(0)\big),\quad  \forall t\in[0,T]. \label{eneA4temp}
    \end{align}
\end{lemma}
\begin{proof}

    Taking the derivative with respect to $t$ on \eqref{eq1}--\eqref{eq3} and then multiplying the resulting equations by $\pt_{t}\varphi$, $\ep^2\pt_{t}\psi$ and $\pt_{t}\theta$, respectively, we obtain, by using the method of integration by parts, that
    \begin{align}
    	&\int\frac{1}{2}\big(|\pt_t\varphi|^2+\ep^2\rho|\pt_t\psi|^2+\rho|\pt_t\theta|^2\big)(t)dx\notag\\
    	&+\ep^2\int_{0}^{t}\big(\mu\|\nabla\pt_t\psi(s)\|^2_{L^2}+(\mu+\mup)\|\div\pt_t\psi(s)\|^2_{L^2}\big)ds
    	+\int_{0}^{t}\kappa\|\nabla\pt_t\theta(s)\|^2_{L^2}ds\notag\\
    	\leq&\int\frac{1}{2}\big(|\pt_t\varphi|^2+\ep^2\rho|\pt_t\psi|^2+\rho|\pt_t\theta|^2\big)(0)dx+ \frac{1}{2}\ep^2\int_{0}^{t}\mu\|\nabla\pt_{t}\psi(s)\|^2_{L^2}ds\notag\\
    	&+ \frac{1}{2}\int_{0}^{t}\kappa\|\nabla\pt_{t}\theta(s)\|^2_{L^2}ds+C(A_0(t)+A_1(t)+A_2(t)),   \label{T0}
    \end{align}
    where we have used \eqref{ob0}, \eqref{infW} and the facts that
    \begin{align}
    	\int \pt_t\varphi (u\cdot\nabla\pt_t\varphi) dx
    	=&-\int \frac{1}{2}\div\psi |\pt_t\varphi|^2 dx
    	\leq C\|\nabla\psi\|_{L^\infty}\|\pt_{t}\varphi\|^2_{L^2}
    	\leq C\ep^{-2}\|\pt_{t}\varphi\|^2_{L^2},    \notag\\
    	\ep^2\int (\pt_{t}(\rho\psi)\cdot\nabla\psi)\cdot\pt_{t}\psi dx
    	\leq& C\ep^2(\|\pt_{t}\varphi \|_{L^2}+\|\pt_{t}\psi \|_{L^2})\|\nabla\psi \|_{L^\infty}
    	\|\pt_{t}\psi \|_{L^2}\notag\\
    	\leq& C\ep(\|\pt_{t}\varphi \|^2_{L^2}+\|\pt_{t}\psi \|^2_{L^2}),\notag
    \end{align} 
and
    \begin{align}
    	&\int\big(\rho\pt_{t}\div\psi\pt_t\varphi+ \nabla\pt_t P(\rho,\fT)\cdot\pt_t\psi
    	+P(\rho,\fT)\pt_{t}\div\psi\pt_{t}\theta\big)dx  \notag\\
    	=&\int\big((\rho-\fT)\pt_{t}\div\psi\pt_t\varphi+\rho(\fT-1)\pt_{t}\div\psi\pt_{t}\theta\big)dx \notag\\
    	\leq &(\|\rho-1\|_{L^\infty}+\|\fT-1\|_{L^\infty})\|\pt_{t}\varphi\|_{L^2}\|\pt_{t}\div\psi\|_{L^2} 
    	+\|\rho\|_{L^\infty}\|\fT-1\|_{L^\infty}\|\pt_{t}\theta\|_{L^2}\|\pt_{t}\div\psi\|_{L^2} \notag\\
    	\leq& C\ep\|\pt_{t}\varphi\|_{L^2}\|\pt_{t}\div\psi\|_{L^2}
    	+C\ep\|\pt_{t}\theta\|_{L^2}\|\pt_{t}\div\psi\|_{L^2}\notag\\
    	\leq& \frac{1}{4}\ep^2\mu\|\nabla\pt_{t}\psi\|^2_{L^2}
    	+C\big(\|\pt_{t}\varphi\|^2_{L^2}+\|\pt_{t}\theta\|^2_{L^2}\big).\notag
    \end{align} 
    Thus, \eqref{T0}, together with \eqref{eneA0} and \eqref{eneA12}, gives \eqref{eneA3}.
 
    Next, we rewrite \eqref{eq2} as the Lam\'{e} system
    \begin{equation}
    	\left.\begin{cases}
    		-\mu\Delta \psi-(\mu+\mup)\nabla\div\psi
    		=-\ep^{-2}\nabla P(\rho,\fT)-\rho\big(\pt_t\psi+u\cdot\nabla \psi+\psi\cdot\nabla\tu\big),
    		\quad\text{in}\,\,\Omega,\\
    		\psi|_{x_2=0,1}=0,
    	\end{cases}\label{elliptic}
    	\right.
    \end{equation}
    Applying \Cref{Lame lemma} to the boundary value problem \eqref{elliptic}, we obtain, by using \eqref{infW} and \eqref{eneA3}, that
    \begin{align}
    	\ep^2\|\nabla^2\psi(t)\|^2_{L^2}
    	\leq& C\big(\ep^{-2}\|\varphi(t)\|^2_{H^1}+\ep^{-2}\|\theta(t)\|^2_{H^1}
    	 +\ep^2\|\psi(t)\|^2_{H^1}+ \ep^2\|\pt_{t}\psi(t)\|^2_{L^2}\big)\notag\\
    	 \leq& C\big(\ep^2A_3(0)+A_1(0)+A_2(0)+\ep^{-2}A_0(0)\big),\quad\quad  \forall t\in[0,T].
    \end{align}  
    Similarly, applying the elliptic theory \cite{ADN1959,ADN1964} to the boundary value problem  \eqref{ellptictheta}, we have 
      \begin{align}
    	\|\nabla^2\theta(t)\|^2_{L^2}
    	\leq& C\big(\|\psi(t)\|^2_{H^1}+\|\pt_{t}\theta(t)\|^2_{L^2}
    	+\|\nabla\theta(t)\|^2_{L^2}\big)\notag\\
    	\leq& C\big(\ep^2A_3(0)+A_1(0)+A_2(0)+\ep^{-2}A_0(0)\big), \quad\quad  \forall t\in[0,T].
    \end{align} 
The proof is completed.
\end{proof}

Next, we derive the $\ep$-weighted $H^2$-type estimates on $(\varphi,\psi,\theta)$ as follows.
\begin{lemma}\label{lemmaH2}
	Suppose that \eqref{chi} holds, and suppose that $\Re<\Re_0$ and $\ep\leq\vep_2$ as in Lemma \ref{lemma3}. Then, there exist a positive constant $N_1$ depending only on $\Re$, $\Pr$, $\frac{\mup}{\mu}$ and $\gamma$, such that if $\Nhat\leq N_1$, then we have
\begin{align}
	\ep^{2}A_4(t)\leq C\big(\ep^2A_3(0)+\ep^2A_4(0)+A_1(0)+A_2(0)+\ep^{-2}A_0(0)\big), \quad\quad \forall t\in[0,T]. \label{eneA4}
\end{align}
\end{lemma}
\begin{proof}
	Similar to the proof of \Cref{lemma2}, we first apply the operator $\pt^2_{1}$ on both \eqref{eqf1} and \eqref{eqf2} to get
	\begin{align}
		&\pt_t\pt^2_{1}\varphi+u\cdot\nabla\pt^2_{1}\varphi+\div(\pt^2_{1}\psi)
		=\pt^2_{1}f_1+T_1,   \label{eq1t}
			\end{align}
		and
		\begin{align}
	&\rho\big(\pt_t\pt^2_{1}\psi+u\cdot\nabla\pt^2_{1}\psi+\pt^2_{1}\psi\cdot\nabla\tu\big)
	+\ep^{-2}\nabla\big(\pt^2_{1}\varphi+\pt^2_{1}\theta\big)
	-\mu\Delta\pt^2_{1}\psi-(\mu+\mup)\nabla\div\pt^2_{1}\psi\notag\\
	=&\ep^{-2}\nabla\pt^2_{1}f_2+T_2+T_3,\label{eq2t}
	\end{align}
	where
	\begin{align}
		&T_1=-\pt^2_{1}\psi\cdot\nabla\varphi-2\pt_{1}\psi\cdot\nabla\pt_{1}\varphi, \quad
		T_2=-\pt^2_{1}\varphi(\pt_t\psi+u\cdot\nabla\psi+\psi\cdot\nabla\tu), \notag\\
		&T_3=-2\pt_{1}\varphi(\pt_t\pt_{1}\psi+u\cdot\nabla\pt_{1}\psi+\pt_{1}\psi\cdot\nabla\psi). \notag
	\end{align}
	 Adding \eqref{eq1t} multiplied by $\ep^{-2}\pt^2_{1}\varphi$ to \eqref{eq2t} multiplied by $\pt^2_{1}\psi$, and then applying the method of integration by parts to the resulting equation, we obtain, by using \eqref{N}, \eqref{eneA0}, \eqref{eneA12}, \eqref{eneA3} and the Young inequality, that
    \begin{align}
    	&\|(\ep^{-1}\pt^2_{1}\varphi,\pt^2_{1}\psi)(t)\|^2_{L^2}
    	+\int_{0}^{t}\big(\|\nabla\pt^2_{1}\psi(s)\|^2_{L^2}
    	+\|\pt^2_{1}\div\psi(s)\|^2_{L^2}\big)ds\notag\\
    	\leq&C\|(\ep^{-1}\pt^2_{1}\varphi,\pt^2_{1}\psi)(0)\|^2_{L^2}
    	+C\big(\ep^{2}A_3(t)+A_1(t)+A_2(t)+\ep^{-2}A_0(t)\big)\notag\\
    	\leq&C\|(\ep^{-1}\pt^2_{1}\varphi,\pt^2_{1}\psi)(0)\|^2_{L^2}
    	+C\big(\ep^{2}A_3(0)+A_1(0)+A_2(0)+\ep^{-2}A_0(0)\big) ,    \label{enetan}
    \end{align}
    provided that $\widehat{N}$ is small enough (but independent of $\ep$ and $T$). In deriving \eqref{enetan}, we have used the fact that
	\begin{align}
		\int_{0}^{t}\int\pt^2_{1}\varphi\pt_t\psi\cdot\pt^2_{1}\psi dxds
		\leq& \int\|\pt^2_{1}\varphi(s)\|_{L^2}\|\pt_t\psi(s)\|_{L^4}\|\pt^2_{1}\psi(s)\|_{L^4}ds\notag\\
		\leq&  \int_{0}^{t}\|\pt^2_{1}\varphi(s)\|^2_{L^2}\|\nabla\pt^2_{1}\psi(s)\|^2_{L^2}ds
		+C\int_{0}^{t}\|\nabla\pt_t\psi(s)\|^2_{L^2}ds,\notag\\
		\leq& \big(\sup_{s\in [0,t]}\|\pt^2_{1}\varphi(s)\|^2_{L^2} \big)\int_{0}^{t}\|\nabla\pt^2_{1}\psi(s)\|^2_{L^2}ds
		+C\int_{0}^{t}\|\nabla\pt_t\psi(s)\|^2_{L^2}ds\notag\\
		\leq& N(t)\int_{0}^{t}\|\nabla\pt^2_{1}\psi(s)\|^2_{L^2}ds
		+C\int_{0}^{t}\|\nabla\pt_t\psi(s)\|^2_{L^2}ds,
	\end{align} 
	and other terms are treated similarly.
	
		Second, applying the operator $\pt_{1}$ on \eqref{eqnormal1} and \eqref{eqnormal2}, and then multiplying the resulting equations by $(2\mu+\mup)\pt_{1}\pt_{2}\varphi$ and $\rho^{-1}\pt_{1}\pt_{2}\varphi$, respectively, we obtain, by using \eqref{N}, \eqref{eneA0}, \eqref{eneA12}, \eqref{eneA3}, \eqref{enetan} and the Young inequality, that
	\begin{align}
	&\|\pt_{1}\pt_{2}\varphi(t)\|^2_{L^2}+\ep^{-2}\int_{0}^{t}\|\pt_{1}\pt_{2}\varphi(s)\|^2_{L^2}ds\notag\\
	\leq&C\|\pt_{1}\pt_{2}\varphi(0)\|^2_{L^2}+ C\ep^{2}\int_{0}^{t}\big(\|\nabla\pt^2_{1}\psi(s)\|^2_{L^2}+\|\psi(s)\|^2_{H^2}+\|\pt_{t}\psi(s)\|^2_{H^1}\big)ds\notag\\
	&+C\ep^{-2}\int_{0}^{t}\big(\|\varphi(s)\|^2_{H^1}+\|\theta(s)\|^2_{H^2}\big)ds\notag\\
	\leq&C\|\nabla\pt_{1}\varphi(0)\|^2_{L^2}+C\ep^2\|\pt^2_{1}\psi(0)\|^2_{L^2}
	+C\big(\ep^2A_3(0)+A_1(0)+A_2(0)+\ep^{-2}A_0(0)\big),   \label{enenor1}
	\end{align}
    provided that $\widehat{N}$ is small enough. Consequently, we derive from \eqref{N} and \eqref{eqnormal3} that 
    \begin{align}
    	(2\mu+\mup)\pt_{1}\pt_{2}\div\psi
    	=&O(\ep^{-2})(|\pt_{1}\pt_{2}\varphi|+|\pt_{1}\pt_{2}\theta|+|\nabla\varphi|+|\nabla\theta|)\notag\\
    	&+O(1)(|\pt_{1}\pt_t\psi|+|\nabla^2\psi|+|\nabla\psi|+|\psi|+|\nabla\pt^2_{1}\psi|),\notag
    \end{align}
	which, together with \eqref{eneA12}, \eqref{enetan} and \eqref{enenor1}, leads to
	\begin{align}
		\int_0^t\|\nabla\pt_{1}\div\psi(s)\|^2_{L^2}ds
		\leq& C\big(\|\nabla\pt_{1}\varphi(0)\|^2_{L^2}
		+\ep^2\|\pt^2_{1}\psi(0)\|^2_{L^2}\big)\notag\\
		&+C\big(\ep^2A_3(0)+A_1(0)+A_2(0)+\ep^{-2}A_0(0)\big).    \label{divtan}
	\end{align}
	
	Third, we consider the following elliptic system obtained from \eqref{Stokes2}:
	\begin{equation}
		\left.\begin{cases}
			\div (\pt_{1}\psi)=\pt_{1}g, \quad\quad\quad\quad\quad\text{in}\,\,\Omega,\\
			-\mu\Delta 	(\pt_{1}\psi)+\ep^{-2}\nabla (\pt_{1}P(\rho,\fT))=\pt_{1}F,
			\quad\quad\quad\text{in}\,\,\Omega,\\
			\pt_{1}\psi|_{x_2=0,1}=0,
		\end{cases}\label{Stokestan}
		\right.
	\end{equation}
    where the definition of $g$ and $F$ can be found in \eqref{gF}. 
    
    It follows from \Cref{Stokes lemma} that
    \begin{align}
    	&\int_{0}^{t}\big(\|\pt_{1}\psi(s)\|^2_{H^2}+\ep^{-4}\|\nabla\pt_{1}P(\rho,\fT)\|^2_{L^2}\big)ds\notag\\
    	\leq &C\int_{0}^{t}\big(\|\nabla\div\pt_{1}\psi(s)\|^2_{L^2}+\|\nabla\pt_{t}\psi(s)\|^2_{L^2}
    	+\|\psi(s)\|^2_{H^2}\big)ds,
    \end{align}
    which, together with \eqref{divtan}, \eqref{eneA3} and \eqref{eneA12}, gives that
     \begin{align}
    	&\int_{0}^{t}\big(\ep^{2}\|\pt_{1}\psi(s)\|^2_{H^2}
    	+\ep^{-2}\|\nabla\pt_{1}\varphi\|^2_{L^2}\big)ds\notag\\
    	\leq &C\ep^{2}\int_{0}^{t}\big(\|\nabla\div\pt_{1}\psi(s)\|^2_{L^2}
    	+\|\nabla\pt_{t}\psi(s)\|^2_{L^2}
    	+\|\psi(s)\|^2_{H^2}\big)ds+C\ep^{-2}\int_{0}^{t}\|\theta\|^2_{H^2}ds. \label{enetan2}
    \end{align}

    Next, applying the operator $\pt_{2}$ on both \eqref{eqnormal1} and \eqref{eqnormal2}, and then multiplying the resulting equations by $(2\mu+\mup)\pt_{2}\pt_{2}\varphi$ and $\rho^{-1}\pt_{2}\pt_{2}\varphi$, respectively, we obtain, by using \eqref{N}, \eqref{eneA0}, \eqref{eneA12}, \eqref{eneA3}, \eqref{enetan}, \eqref{enetan2} and the Young inequality, that
    	\begin{align}
    	&\|\pt^2_{2}\varphi(t)\|^2_{L^2}+\ep^{-2}\int_{0}^{t}\|\pt^2_{2}\varphi(s)\|^2_{L^2}ds\notag\\
    	\leq&C\|\pt^2_{2}\varphi(0)\|^2_{L^2}+ C\ep^{2}\int_{0}^{t}\big(\|\nabla^2\pt_{1}\psi(s)\|^2_{L^2}+\|\psi(s)\|^2_{H^2}+\|\pt_{t}\psi(s)\|^2_{H^1}\big)ds\notag\\
    	&+C\ep^{-2}\int_{0}^{t}\big(\|\varphi(s)\|^2_{H^1}+\|\theta(s)\|^2_{H^2}\big)ds\notag\\
    	\leq&C\|\nabla^2\varphi(0)\|^2_{L^2}+C\ep^2\|\pt^2_{1}\psi(0)\|^2_{L^2}
    	+C\big(\ep^2A_3(0)+A_1(0)+A_2(0)+\ep^{-2}A_0(0)\big),   \label{enenor2}
    \end{align}
     provided that $\widehat{N}$ is small enough. Consequently, we derive from \eqref{N} and \eqref{eqnormal3} that 
     \begin{align}
     	(2\mu+\mup)\pt^2_{2}\div\psi
     	=&O(\ep^{-2})(|\pt^2_{2}\varphi|+|\pt^2_{2}\theta|+|\nabla\varphi|+|\nabla\theta|)\notag\\
     	&+O(1)(|\pt_{2}\pt_t\psi|+|\nabla^2\psi|+|\nabla\psi|+|\psi|
     	+|\nabla^2\pt_{1}\psi|),\notag
     \end{align}
     which, together with \eqref{eneA12}, \eqref{enetan}, \eqref{enenor1} and \eqref{divtan}, leads to
     \begin{align}
     	\int_0^t\|\nabla^2\div\psi(s)\|^2_{L^2}ds
  	\leq& C\big(\|\nabla^2\varphi(0)\|^2_{L^2}+\|\pt^2_{1}\psi(0)\|^2_{L^2}\big)\notag\\
    	&+C\big(\ep^2A_3(0)+A_1(0)+A_2(0)+\ep^{-2}A_0(0)\big).    \label{divnor}
     \end{align}
 Finally, applying \Cref{Stokes lemma} to \eqref{Stokes2}, we obtain, by using \Cref{Poin} and \eqref{eneA0}, \eqref{eneA12}, \eqref{eneA3} and \eqref{divnor}, that
 \begin{align}
 	\int_0^t\big(\ep^{2}\|\psi(s)\|^2_{H^3}+\ep^{-2}\|\varphi(s)\|^2_{H^3}\big)ds
 	\leq& C\big(\|\nabla^2\varphi(0)\|^2_{L^2}+\|\pt^2_{1}\psi(0)\|^2_{L^2}\big)\notag\\
 	&+C\big(\ep^2A_3(0)+A_1(0)+A_2(0)+\ep^{-2}A_0(0)\big). \label{enenor3}
 \end{align}
 Moreover, using the analogous argument as in \Cref{lemma4}, we have
  \begin{align}
 	\int_0^t\ep^{-2}\|\theta(s)\|^2_{H^3}ds
 	\leq C\big(\ep^2A_3(0)+A_1(0)+A_2(0)+\ep^{-2}A_0(0)\big). \label{enetheta}
 \end{align}

   Thus, from \eqref{eneA4temp}, \eqref{enetan}, \eqref{enetan2}, \eqref{enenor2}, \eqref{enenor3} and \eqref{enetheta}, we conclude \eqref{eneA4}.

 The proof is completed.
\end{proof}	
	
	Then, we state $\ep$-weighted $H^3$-type estimates on $(\varphi,\psi,\theta)$ as follows.
	\begin{lemma} \label{lemmaH3}
		Suppose that \eqref{chi} holds, and suppose that $\Re<\Re_0$ and $\ep\leq\vep_2$ as in Lemma \ref{lemma3}. Then, there exists a positive constant $N_2$ depending only on $\Re$, $\Pr$, $\frac{\mup}{\mu}$ and $\gamma$, such that if $\Nhat\leq N_2$, then we have
		\begin{align}
			\ep^{4}A_5(t)\leq C\big(\ep^{4}A_5(0)+\ep^2A_3(0)+\ep^2A_4(0)+A_1(0)+A_2(0)+\ep^{-2}A_0(0)\big), \,\,\forall t\in[0,T]. \label{eneA5}
		\end{align}
	\end{lemma}
   \begin{proof}
   	This lemma can be proved by using a similar argument as in the proof of \Cref{lemmaH2temp,lemmaH2}. The details are omitted here.
   \end{proof}

\section{Proof of the main theorems}
In this section we prove \Cref{main1,main2}. 

\subsection{Proof of \Cref{main1}}
To prove \Cref{main1}, we will follow standard arguments for problems with small data as in \cite{MN1980,MN1983}. Thus, we only give a sketch of proof as follows.

{\bf \hspace{-1em}Proof of \Cref{main1}.}
Let $\Re^\prime$, $\vep^\prime$, $N^\prime$ and $\hat{C}$ be the same as in \Cref{priori}. By \Cref{local}, there exist a time $T_*>0$ and a unique strong solution $(\varphi, \psi,\theta)$ to the initial-boundary value problem \eqref{eq1}--\eqref{eq5} on $(0,T_*)\times\Omega$ such that \eqref{local0} holds. 

Let $N(t)$ be defined by \eqref{N}, and suppose that 
\begin{align}
	N(0)\leq \max\{\frac{1}{4},\frac{1}{8}\hat{C}^{-1}\}N^\prime.  \label{prove2}
\end{align}
Then, due to \eqref{local0} there exists a time $t_1\in (0,T_*]$ such that
\begin{align}
	\sup_{t\in(0,t_1)}N(t)\leq 2N(0) \leq \frac{1}{2}N^\prime. \label{prove1}
\end{align}
Thus, it follows from \eqref{prove1} and \Cref{priori} that 
\begin{align}
	\sup_{t\in(0,t_1)}N(t)\leq \hat{C}N(0),
\end{align}
which, together with \eqref{prove2}, leads to 
\begin{align}
	\sup_{t\in(0,t_1)}N(t)\leq \frac{1}{4}N^\prime.
\end{align}

Next, we can solve the problem \eqref{eq1}--\eqref{eq5} in $t\geq t_1$ with initial data $(\varphi(t_1),\psi(t_1),\theta(t_1))$ again, and by uniqueness we can extend the solution $(\varphi,\psi,\theta)$ to $[0,2t_1]$. Therefore, we can continue the above argument and the same process for $0\leq t\leq nt_1$, $n=2,3,4,\cdots$ and finally obtain a global unique strong solution $(\varphi,\psi,\theta)$ satisfying \eqref{priori} for any $t>0$. Let $(\rho,u,\fT)=(\trho+\varphi,\tu+\psi,\tfT+\theta)$. It can be seen that $(\rho,u,\fT)$ is indeed a global unique strong solution to the original problem \eqref{2DCNS0}--\eqref{initialcon1} such that \eqref{regularity} and \eqref{uniform} hold. 

Finally, the large time behavior \eqref{largetime} can be shown by using the Sobolev embedding theorem and the fact from \eqref{priori} that $(\varphi,\psi,\theta)\in H^1\big([0,\infty);H^2(\Omega)\big)$.

The proof of is completed.                                   \,\,\,$\hspace{25em}\Box$

\subsection{Proof of \Cref{main2}}
Now, we are ready to prove \Cref{main2}.

{\bf \hspace{-1em}Proof of \Cref{main2}.}
The condition \eqref{same} implies that \eqref{chi} holds. Therefore, \Cref{main1} guarantees the global existence of strong solution $(\rho^\ep,u^\ep,\fT^\ep)$ to \eqref{2DCNS0}--\eqref{initialcon1}, which satisfies \eqref{regularity} and \eqref{uniform}, i.e., it holds that
\begin{equation}
	\begin{aligned}
		\|\rho^\ep-\trho\|_{L^2}=O(\ep^2),
		\quad \|u^\ep-\tu\|_{L^2}=O(\ep),
		\quad \|\fT^\ep-\tfT\|_{L^2}=O(\ep^2).
	\end{aligned}  \label{converge1}
\end{equation}
Moreover, it is observed from \eqref{Couette1} and \eqref{same} that 
\begin{equation}
	\begin{aligned}
		\|\trho-1\|_{L^2}=O(\ep^2),
		\quad \|\tfT-1\|_{L^2}=O(\ep^2).\label{converge2}
	\end{aligned} 
\end{equation}
Combining \eqref{converge1} and \eqref{converge2} gives \eqref{converge0}.

The proof is completed.
\,\,\,$\hspace{25em}\Box$\\


\noindent {\bf Acknowledgments:}
The authors are supported by the NSFC (Grants 12071044 and 12131007).


\begin{thebibliography}{99}
\bibitem{ADN1959}
S. Agmon, A. Douglis and L. Nirenberg. Estimates near the boundary for solutions of elliptic partial differential equations satisfying general boundary conditions. I, Comm. Pure Appl. Math., 12: 623--727, 1959.

\bibitem{ADN1964}
S. Agmon, A. Douglis, and L. Nirenberg. Estimates near the boundary for solutions of elliptic partial differential equations satisfying general boundary conditions. II, Comm. Pure Appl. Math., 17: 35--92, 1964.
	
\bibitem{A2006}	
T. Alazard, Low Mach number limit of the full Navier-Stokes equations. Arch. Ration. Mech. Anal., 180:1--73, 2006.	
	
\bibitem{B1995}
H. Bessaih, Limite de modeles de fluides compressibles, Protugaliae mathematica, 52: 441--463, 1995.

\bibitem{CYLLL2016}
J. Cheng, X. Yang, X. Liu, T. Liu and H. Luo, A direct discontinuous galerkin method for the compressible navier–stokes equations on arbitrary grids, J. Comput. Phys., 327: 484--502, 2016.

\bibitem{Dafermos2016}
C. M. Dafermos, Hyperbolic Conservation Laws in Continuum Physics, 4th ed., Grundlehrender Mathematischen Wissenschaften [Fundamental Principles of Mathematical Sciences], vol. 325, Springer-Verlag, Berlin, 2016.

\bibitem{DGLM1999}
B. Desjardins, E. Grenier, P. L. Lions and N. Masmoudi, Incompressible limit for solutions of the isentropic Navier-Stokes equations with Dirichlet boundary conditions, J. Math. Pures. Appl., 78: 461--471, 1999.

\bibitem{DEH1994}
P. W. Duck, G. Erlebacher and M. Y. Hussaini, On the linear stability of compressible plane Couette flow, J. Fluid Mech., 258: 131--165, 1994.

\bibitem{FJX2023}
X. Fan, Q. Ju and J. Xu, Long time existence of slightly compressible Navier-Stokes equations in bounded domains with non-slip boundary condition, submitted, 2023.

\bibitem{Galdi2011}
G. P. Galdi. \textit{An introduction to the mathematical theory of the Navier-Stokes equations.} Second. Springer Monographs in Mathematics. Springer, New York 2011.

\bibitem{HWW2017}
F. Huang, T. Wang and Y. Wang, Diffusive wave in the low Mach limit for compressible Navier–Stokes equations, Adv. Math., 319: 348--395, 2017. 

\bibitem{IP1998}
G. Iooss and M. Padula, Structure of the linearized problem for compressible parallel fluid flows. Ann. Univ. Ferrara, Sez. VII 43: 157--171, 1998.

\bibitem{JO2011}
S. Jiang and Y. Ou, Incompressible limit of the non-isentropic Navier-Stokes equations with well-prepared initial data in three-dimensional bounded domains, J. Math. Pures Appl., 96: 1--28, 2011.

\bibitem{JO2022}
Q. Ju and Y. Ou, Low Mach number limit of Navier-Stokes equations with large temperature variations in bounded domains, J. Math. Pures Appl., 164: 131--157, 2022.

\bibitem{K2011}
Y. Kagei, Asymptotic behavior of solutions of the compressible Navier–Stokes equation around the plane Couette flow, J. Math. Fluid Mech., 13: 1--31, 2011.

\bibitem{K2011JDE}
Y. Kagei, Global existence of solutions to the compressible Navier–Stokes equation around parallel flows, J. Differential Equations, 251: 3248--3295, 2011.

\bibitem{K2012}
Y. Kagei, Asymptotic behavior of solutions to the compressible Navier-Stokes equation around a parallel flow, Arch. Rational Mech. Anal., 205: 585--650, 2012.

\bibitem{KN2015}
Y. Kagei and T. Nishida, Instability of plane Poiseuille flow in viscous compressible gas, J. Math. Fluid Mech., 17: 129--143, 2015.

\bibitem{KN2019}
Y. Kagei and T. Nishida, Traveling waves bifurcating from plane Poiseuille flow of the compressible Navier–Stokes equation, Arch. Ration. Mech. Anal., 231: 1--44, 2019.

\bibitem{KT2020}
Y. Kagei and Y. Teramoto, On the spectrum of the linearized operator around compressible Couette flows between two concentric cylinders, J. Math. Fluid Mech., 20: Paper No. 21, 2020.

\bibitem{KW2010}
R. Kannan and Z. Wang, The direct discontinuous galerkin (ddg) viscous flux scheme for the high order spectral volume method, Comput. Fluids 39: 2007--2021, 2010.

\bibitem{KM1981}
S. Klainerman and A. Majda, Singular perturbations of quasilinear hyperbolic systems with large parameters and the incompressible limit of compressible fluids, Comm. Pure Appl. Math., 34: 481--524, 1981.

\bibitem{KM1982}
S. Klainerman and A. Majda, Compressible and incompressible fluids, Commun. Pure Appl. Math., 35: 629--653, 1982.

\bibitem{LZ2017}
H. Li and X. Zhang, Stability of plane couette flow for the compressible Navier-Stokes equations with Navier-slip boundary, J. Differential Equations, 263: 1160--1187, 2017.

\bibitem{Lions1996}
P. L. Lions, Mathematical Topics in Fluid Dynamics. Vol. 1. Incompressible Models, Oxford Univ. Press, London, 1996.

\bibitem{LM1998}
P. L. Lions and N. Masmoudi, Incompressible limit for a viscous compressible fluid, J. Math. Pures Appl., 77: 585--627, 1998.

\bibitem{MDA2008}
M. Malik, J. Dey and M. Alam, Linear stability, transient energy growth, and the role of viscosity stratification in compressible plane Couette flow. Phys. Rev. E, 77: 036322, 2008.

\bibitem{MRS2022}
N. Masmoudi, F. Rousset and C. Sun. Uniform regularity for the compressible Navier-Stokes system with low Mach number in domains with boundaries. J. Math. Pures Appl., 161: 166--215, 2022.

\bibitem{MZ2020}
N. Masmoudi and W. Zhao, Enhanced dissipation for the 2D Couette flow in critical space, Comm. Partial Differential Equations, 45: 1682--1701, 2020.

\bibitem{MN1980}
A. Matsumura and T. Nishida, The initial value problem for the equations of motion of viscous and heat-conductive gases, J. Math. Kyoto Univ. 20(1): 67--104, 1980.

\bibitem{MN1983}
A. Matsumura and T. Nishida, Initial-boundary value problems for the equations of motion of compressible viscous and heat-conductive fluids, Comm. Math. Phys., 89(4): 445--464, 1983.

\bibitem{N1959}
L. Nirenberg, On elliptic partial differential equations, Ann. Scuola Norm. Sup. Pisa, 13: 115--162, 1959.

\bibitem{O2009}
Y. Ou, Incompressible limits of the Navier–Stokes equations for all time, J. Differential Equations, 247: 3295--3314, 2009.

\bibitem{OY2022}
Y. Ou and L. Yang, Incompressible limit of isentropic Navier-Stokes equations with ill-prepared data in bounded domains, SIAM J. Math. Anal., 54: 2948--2989, 2022.

\bibitem{RO2016}
D. Ren and Y. Ou, Incompressible limit and stability of all-time solutions to 3-D full Navier-Stokes equations for perfect gases, Sci. China Math., 59: 1395--1416, 2016.

\bibitem{R1973}
V. A. Romanov, Stability of plane-parallel Couette flow. Funct. Anal. Appl., 7: 137--146, 1973.

\bibitem{S2022}
C. Sun, Uniform regularity in the low Mach number and inviscid limits for the full Navier-Stokes system in domains with boundaries, preprinted in arXiv:2204.09799v2, 2022. 
https://doi.org/10.48550/arXiv.2204.09799.

\bibitem{Veigo1997}
H. B. D. Veigo, A new approach to the $L^2$-regularity theorems for linear stationary nonhomogeneous stokes systems, Portugaliae Mathematica, 54(3): 271--286, 1997.

\bibitem{WZZ2018}
D. Wei, Z. Zhang, and W. Zhao, Linear inviscid damping for a class of monotone shear flow in sobolev spaces, Comm. Pure. Appl. Math., 71: 617--687, 2018.

\bibitem{WZZ2019}
D. Wei, Z. Zhang, and W. Zhao, Linear inviscid damping and vorticity depletion for shear flows, Ann. PDE, 5 (2019), Art. 3, 101.

\bibitem{WZ2021}
D. Wei and Z. Zhang, Transition threshold for the 3D Couette flow in Sobolev space, Comm. Pure Appl.Math., 74: 2398--2479, 2021.

\bibitem{Xu2001}
K. Xu, A gas-kinetic BGK scheme for the Navier-Stokes equations and its connection with artificial dissipation and Godunov method, J. Comput. Phys. 171: 289--335, 2001.

\bibitem{Z2021}
X. Zhai, Linear stability of the Couette flow for the non-isentropic compressible fluid, preprinted in arXiv:2107.03268, 2021. https://doi.org/10.48550/arXiv.2107.03268.

\end{thebibliography}

\end{document}